\documentclass[12pt,a4paper]{article}

\usepackage[verbose]{geometry}
\usepackage[T1]{fontenc} 
\usepackage[utf8]{inputenc}

\ifdefined\ifdraftdoc

\else
  \usepackage{ifdraft}
  \makeatletter
  \let\ifdraftdoc\if@draft
  \makeatother
\fi

\usepackage{needspace}

\usepackage{multirow}
\usepackage[leqno]{amsmath}
\allowdisplaybreaks[4]
\usepackage{amsthm}
\usepackage{amssymb}
\usepackage{amsfonts}

\renewcommand{\qedsymbol}{$\blacksquare$}
\newcommand{\openqedsymbol}{$\square$}
\newcommand{\exercisesymbol}{$\lozenge$}

\newcommand{\openproof}{\renewcommand{\qedsymbol}{\openqedsymbol}}
\newcommand{\exerproof}{\renewcommand{\qedsymbol}{\exercisesymbol}}

\usepackage{mathtools}

\usepackage[british]{babel}

\usepackage[autostyle]{csquotes}
\usepackage{hyphenat}

\usepackage{etoolbox}
\usepackage{xstring}

\usepackage[shortlabels]{enumitem}

\usepackage[dvipsnames,hyperref]{xcolor}

\usepackage{tikz}
\usepackage{tikz-cd}

\usetikzlibrary{arrows, shapes.geometric}
\tikzset{baseline=(current bounding box.center)}

\let\LiningNumbers\relax

\DeclareTextFontCommand{\textrmup}{\rmfamily\mdseries\upshape}
\DeclareTextFontCommand{\textbfup}{\rmfamily\bfseries\upshape}
\DeclareTextFontCommand{\textttup}{\ttfamily\mdseries\upshape}
\DeclareTextFontCommand{\textsfup}{\sffamily\mdseries\upshape}

\setlist[enumerate]{font=\upshape}

\DeclareMathAlphabet{\mathpzc}{OT1}{pzc}{m}{it}

\usepackage{microtype}

\usepackage{xspace}

\newcommand{\hairspace}{\ifmmode\mskip1mu\else\kern0.08em\fi}

\makeatletter
\def\abbrdot{\@ifnextchar.{}{.\@\xspace}}
\makeatother
\newcommand{\etc}{etc\abbrdot}
\newcommand{\eg}{e.g\abbrdot}
\newcommand{\ie}{i.e\abbrdot}
\newcommand{\resp}{resp\abbrdot}

\newcommand{\Sect}{\S~}

\newcommand{\strong}[1]{\textbfup{#1}}

\renewcommand{\implies}{\ifmmode\Longrightarrow\else$\Rightarrow$\expandafter\xspace\fi}
\renewcommand{\iff}{\ifmmode\Longleftrightarrow\else$\Leftrightarrow$\expandafter\xspace\fi}

\makeatletter
\let\internal@prime\prime
\renewcommand{\prime}{\ifmmode\internal@prime\else$\sp{\internal@prime}$\expandafter\xspace\fi}
\makeatother

\ifdefined\nequiv

\else
  \def\nequiv{\not\equiv}
\fi

\newcommand{\blank}{\mathord{-}}

\newcommand{\embedinto}{\hookrightarrow}
\newcommand{\hoto}{\Rightarrow}

\DeclareMathSymbol{.}{\mathpunct}{letters}{"3A}
\DeclareMathSymbol{:}{\mathrel}{operators}{"3A}

\DeclareMathOperator{\ob}{ob}

\DeclareMathOperator{\iso}{iso}

\newcommand{\Ho}[1][]{\mathop{\mathrm{Ho}\ifstrempty{#1}{}{\sb{#1}}}}

\newcommand{\id}{\mathrm{id}}

\AtBeginDocument{}

\newcommand*{\lr}[3]{\mathopen{}\left#1{#2}\right#3\mathclose{}}
\newcommand*{\lmr}[5]{\mathopen{}\left#1{#2}\,\middle#3\,{#4}\right#5\mathclose{}}

\newcommand*{\ordlr}[3]{\mathord{\left#1{#2}\right#3}}
\newcommand*{\argp}[1]{\ifstrempty{#1}{}{\lr({#1})}}
\newcommand*{\argptwo}[2]{\lr({#1, #2})}
\newcommand*{\argptwotwo}[4]{\lr({#1, #2; #3, #4})}
\newcommand*{\argb}[1]{\ifstrempty{#1}{}{\lr[{#1}]}}
\newcommand*{\parens}[1]{\lr({#1})}

\newcommand*{\bracket}[1]{\lr[{#1}]}
\newcommand*{\tuple}[1]{\ordlr({#1})}

\newcommand*{\prodtuple}[1]{\lr<{#1}>}

\newcommand*{\setbuilder}[2]{\mathord{\lmr\{{#1}|{#2}\}}}
\newcommand*{\seqbuilder}[2]{\mathord{\lmr({#1}|{#2})}}

\newcommand*{\ul}[1]{\smash{\underline{#1}}\vphantom{#1}}

\makeatletter
\newcommand*{\optargp}{\new@ifnextchar\bgroup{\argp}{}}
\newcommand*{\optargptwo}{\new@ifnextchar\bgroup{\expandafter\optargptwo@first}{}}
\newcommand*{\optargptwo@first}[1]{\new@ifnextchar\bgroup{\argptwo{#1}}{\argp{#1}}}
\newcommand*{\optargptwotwo}{\new@ifnextchar\bgroup{\expandafter\optargptwotwo@first}{}}
\newcommand*{\optargptwotwo@first}[2]{\new@ifnextchar\bgroup{\argptwotwo{#1}{#2}}{\argptwo{#1}{#2}}}

\newcommand*{\sbtwo}[2]{\sb{\lr({{#1}, {#2}})}}

\newcommand*{\optsb}{\new@ifnextchar\bgroup{\sb}{}}
\newcommand*{\optsp}{\new@ifnextchar\bgroup{\sp}{}}
\newcommand*{\optsbtwo}{\new@ifnextchar\bgroup{\expandafter\optsbtwo@first}{}}
\newcommand*{\optsbtwo@first}[1]{\new@ifnextchar\bgroup{\sbtwo{#1}}{\sb{#1}}}

\newcommand*{\set}[1]{\new@ifnextchar\bgroup{\setbuilder{#1}}{\ordlr\{{#1}\}}}
\newcommand*{\seq}[1]{\new@ifnextchar\bgroup{\seqbuilder{#1}}{\tuple{#1}}}
\makeatother

\ifdefined\hash

\else
  \def\hash{\#}
\fi

\ifdefined\powerset
  \undef\powerset
\else

\fi
\newcommand*{\powerset}{\mathscr{P} \optargp}

\newcommand{\upJ}{\mathrm{J}}
\newcommand{\upK}{\mathrm{K}}
\newcommand{\upM}{\mathrm{M}}
\newcommand{\uph}{\mathrm{h}}
\newcommand{\upv}{\mathrm{v}}
\newcommand{\mbfI}{\mathbf{I}}

\newcommand*{\smashsp}[2]{{#1}\sp{\smash{#2}}}

\newcommand*{\op}[1]{\smashsp{#1}{\mathrm{op}}}

\makeatletter
\newcommand*{\ctchoice}[2]{%
  \ifdef{\ct@homstyle}%
  {#2}%
  {#1}}
\newcommand*{\ctchoicetwo}[5]{%
  \ifdef{\ct@homcatstyle}%
  {#2}{%
  \ifdef{\ct@transfstyle}%
  {#3}{%
  \ifdef{\ct@catstyle}%
  {#4}{%
  \ifdef{\ct@homstyle}%
  {#5}{%
  {#1}}}}}}
  
\newcommand*{\DefineCategory}[2]{\DeclareRobustCommand{#1}{\ctchoice{\mathbf{#2}}{\mathbf{#2}}}}
\newcommand*{\DefineBicategory}[2]{\DeclareRobustCommand{#1}{\ctchoicetwo{\mathfrak{#2}}{\mathfrak{#2}}{\mathfrak{#2}}{\mathbf{#2}}{\mathbf{#2}}}}

\newcommand*{\cat}[1]{%
  \begingroup%
  \def\ct@catstyle{cat}%
  #1%
  \endgroup}
\newcommand*{\bicat}[1]{%
  \begingroup
  \def\ct@bicatstyle{bicat}%
  #1%
  \endgroup}
\newcommand*{\Hom}[1][]{%
  \begingroup%
  \def\ct@homstyle{hom}%
  \ifstrempty{#1}{\mathrm{Hom}}{#1}%
  \endgroup%
  \optargptwo}
\newcommand*{\HHom}[1][]{%
  \begingroup%
  \def\ct@homcatstyle{homcat}%
  \ifstrempty{#1}{\mathbf{Hom}}{#1}%
  \endgroup
  \optargptwo}
\newcommand*{\hpty}[1][]{%
  \begingroup%
  \def\ct@transfstyle{transf}%
  \ifstrempty{#1}{\mathrm{Nat}}{#1}%
  \endgroup
  \optargptwotwo}

\makeatother

\DefineCategory{\Set}{Set}

\DefineCategory{\Simplex}{\Delta}
\DefineCategory{\SSet}{sSet}
\DefineCategory{\SSSet}{ssSet}

\DefineBicategory{\Qcat}{Qcat}
\DefineBicategory{\CSS}{CSS}
\newcommand{\Cat}{\ctchoicetwo{\mathfrak{Cat}}{\mathbf{Fun}}{\mathbf{Fun}}{\mathbf{Cat}}{\mathrm{Fun}}}
\newcommand{\Grpd}{\ctchoicetwo{\mathfrak{Grpd}}{\mathbf{Fun}}{\mathbf{Fun}}{\mathbf{Grpd}}{\mathrm{Fun}}}
\newcommand{\SCat}{\ctchoicetwo{\mathfrak{SCat}}{\mathbf{SFun}}{\mathbf{SFun}}{\mathbf{SCat}}{\mathrm{SFun}}}

\newcommand*{\Nv}[1][]{\mathbf{N}\sp{#1}\optargp}
\newcommand*{\nv}[1][]{\mathrm{N}\sp{#1}\optargp}

\newcommand*{\xHom}[3][]{\ifstrempty{#1}{\ordlr[{{#2}, {#3}}]}{\ordlr[{{#2}, {#3}}]\sb{#1}}}
\newcommand*{\ulHom}[1][]{\ifstrempty{#1}{\ul{\Hom}}{\ul{\Hom[#1]}}\optargptwo}

\newcommand*{\Func}[2]{\xHom{#1}{#2}}

\makeatletter
\renewcommand*{\@makefnmark}{\hbox{\textsuperscript{\normalfont [\LiningNumbers\@thefnmark]}}}
\renewcommand*{\@makefntext}[1]{\parindent 1.0em\noindent\ifdefempty{\@thefnmark}{}{\hb@xt@-1.0em{\hss \normalfont [\LiningNumbers \@thefnmark]}\hspace{1.0em}}#1}
\@addtoreset{footnote}{chapter}
\newcommand{\footpar}[1]{\gdef\@thefnmark{}\@footnotetext{#1}}

\newcommand{\hangsecnum}{\def\@seccntformat##1{\llap{\csname the##1\endcsname\quad}}}
\makeatother

\usepackage[natbib,backend=biber,style=authoryear-comp,bibstyle=authoryear]{biblatex}

\DeclareSortingScheme{mysort}{
  \sort{
    \field{presort}
  }
  \sort[final]{
    \field{sortkey}
  }
  \sort{
    \name{sortname}
    \name{author}
    \name{editor}
    \name{translator}
    \field{sorttitle}
    \field{title}
  }
  \sort{
    \field{sortyear}
    \field{year}
  }
  \sort{
    \field{month}
  }
  \sort{
    \field{day}
  }
  \sort{
    \field{journal}
  }
  \sort{
    \field[padside=left,padwidth=4,padchar=0]{volume}
    \literal{0000}
  }
  \sort{
    \field{series}
  }
  \sort{
    \field[padside=left,padwidth=4,padchar=0]{number}
    \literal{0000}
  }
  \sort{
    \field{sorttitle}
    \field{title}
  }
}

\usepackage[colorlinks,linkcolor=Blue,citecolor=Blue,urlcolor=blue,pdfpagelabels,pdffitwindow=false]{hyperref}

\usepackage{thmtools}

\declaretheorem[style=plain,parent=section,title=Theorem,refname=theorem,Refname=Theorem]{thm}
\declaretheorem[style=plain,sibling=thm,title=Proposition,refname=proposition,Refname=Proposition]{prop}
\declaretheorem[style=plain,sibling=thm,title=Lemma,refname=lemma,Refname=Lemma]{lem}
\declaretheorem[style=plain,sibling=thm,title=Corollary,refname=corollary,Refname=Corollary]{cor}

\declaretheorem[style=definition,sibling=thm,title=Definition,refname=definition,Refname=Definition]{dfn}

\declaretheorem[style=definition,sibling=thm,title=Example,refname=example,Refname=Example]{example}
\declaretheorem[style=definition,sibling=thm,title=Examples,refname=examples,Refname=Examples]{examples}

\declaretheoremstyle[style=remark,headfont=\scshape]{remark}
\declaretheorem[style=remark,sibling=thm,title=Remark,refname=remark,Refname=Remark]{remark}

\declaretheorem[style=plain,numbered=no,title=Theorem]{thm*}
\declaretheorem[style=plain,numbered=no,title=Proposition]{prop*}
\declaretheorem[style=remark,numbered=no,title=Remark]{remark*}
\declaretheorem[style=definition,numbered=no,title=Example]{example*}

\makeatletter
\@counteralias{numpar}{thm}

\newcommand*{\makenumpar}[1][]{%
  \par%
  \refstepcounter{numpar}%
  {\normalfont{\bfseries\textparagraph}\hspace{0.5em}\thenumpar\ifstrempty{#1}{}{ ({\normalfont #1})}.} }
\makeatother

\AtBeginDocument{%
\hyphenation{pull-back pull-backs push-out push-outs co-pro-duct co-pro-ducts co-com-plete sub-ob-ject sub-ob-jects more-over quasi-com-pact de-gen-er-ate pro-vided Gro-then-dieck Law-vere mod-ule sub-al-ge-bra fi-brant co-fi-brant fi-brant-ly co-fi-brant-ly bi-cat-e-go-ry bi-cat-e-gor-ic-al bi-cat-e-gor-ic-al-ly qua-si-cat-e-go-ry qua-si-cat-e-go-ries qua-si-cat-e-go-ric-al}%
}

\renewcommand*{\bibnamedash}{}
\newcommand*{\authoryearpunct}{\hspace*{-\bibhang}\bibsentence}

\makeatletter

\appto\citesetup{\LiningNumbers}

\renewbibmacro*{cite:label}{%
  \iffieldundef{label}
    {\printtext[bibhyperref]{\printfield[citetitle]{labeltitle}}}
    {\printtext[bibhyperref]{\printfield{label}}}}

\renewbibmacro*{date+extrayear}{%
  \iffieldundef{shorthand}
    {\iffieldundef{labelyear}
      {}
      {\printtext[brackets]{\printfield{labelyear}\printfield{extrayear}}}\addspace}
    {\printtext[brackets]{\printfield{shorthand}}}}

\renewbibmacro*{author}{%
  \ifboolexpr{
    test \ifuseauthor
    and
    not test {\ifnameundef{author}}
  }
    {\usebibmacro{bbx:dashcheck}
       {\bibnamedash}
       {\usebibmacro{bbx:savehash}%
        \printnames{author}%
    \newline\setunit{\authoryearpunct}}%
     \iffieldundef{authortype}
       {}
       {\usebibmacro{authorstrg}%
    \setunit{\addspace}}}%
    {\global\undef\bbx@lasthash
     \usebibmacro{labeltitle}%
     \setunit*{\addspace}}%
  \usebibmacro{date+extrayear}}

\renewbibmacro*{bbx:editor}[1]{%
  \ifboolexpr{
    test \ifuseeditor
    and
    not test {\ifnameundef{editor}}
  }
    {\usebibmacro{bbx:dashcheck}
       {\bibnamedash}
       {\printnames{editor}%
    \newline\setunit{\authoryearpunct}%
    \usebibmacro{bbx:savehash}}%
     \usebibmacro{#1}%
     \clearname{editor}%
     \setunit{\addspace}}%
    {\global\undef\bbx@lasthash
     \usebibmacro{labeltitle}%
     \setunit*{\addspace}}%
  \usebibmacro{date+extrayear}}

\renewbibmacro*{bbx:translator}[1]{%
  \ifboolexpr{
    test \ifusetranslator
    and
    not test {\ifnameundef{translator}}
  }
    {\usebibmacro{bbx:dashcheck}
       {\bibnamedash}
       {\printnames{translator}%
    \newline\setunit{\authoryearpunct}%
    \usebibmacro{bbx:savehash}}%
     \usebibmacro{translator+othersstrg}%
     \clearname{translator}%
     \setunit{\addspace}}%
    {\global\undef\bbx@lasthash
     \usebibmacro{labeltitle}%
     \setunit*{\addspace}}%
  \usebibmacro{date+extrayear}}

\makeatother

\appto\bibsetup{\raggedright}
\appto\bibfont{\LiningNumbers\small}

\usepackage{titling}
\usepackage{fancyhdr}

\hangsecnum

\title{The homotopy bicategory of $\tuple{\infty, 1}$-categories}
\author{Zhen~Lin Low}
\date{2 November 2013}

\begin{document}

\maketitle
\footpar{Department of Pure Mathematics and Mathematical Statistics, University of Cambridge, Cambridge, UK. \textsc{E-mail address}: \texttt{Z.L.Low@dpmms.cam.ac.uk}}
% !TEX root = CSS-2cat.tex
\begin{abstract}
Evidence is given for the correctness of the Joyal--Riehl--Verity construction of the homotopy bicategory of the $\tuple{\infty, 2}$-category of $\tuple{\infty, 1}$-categories; in particular, it is shown that the analogous construction using complete Segal spaces instead of quasicategories yields a bicategorically equivalent 2-category.
\end{abstract}

\section*{Introduction}

In recent years, it has become accepted that an $\tuple{\infty, 1}$-category should be something like a category weakly enriched in $\infty$-groupoids, where `$\infty$-groupoid' refers to weak homotopy types, and by `weakly enriched' we mean that composition is unital and associative only up to coherent homotopy. There has been a proliferation of proposed definitions making this precise: for example, simplicially enriched categories, Segal categories, complete Segal spaces, quasicategories, and relative categories. Fortunately, the categories of such structures (with the appropriate notion of equivalence) can each be embedded in a Quillen model category, and the model categories so obtained have been shown to be Quillen-equivalent: see \citep{Bergner:2007a}, \citep{Joyal-Tierney:2007}, and \citep{Barwick-Kan:2012}. In short, the aforementioned theories of $\tuple{\infty, 1}$-categories have the same homotopy theory.

However, the present writer contends that it is not enough to have presentations of the homotopy theory of $\tuple{\infty, 1}$-categories: we would like to also have presentations of its formal category theory. To illustrate, consider a homotopy pullback square in a simplicially enriched category $\mathcal{C}$: is it the case that the corresponding square in the quasicategorical incarnation of $\mathcal{C}$ is a pullback square in the sense of Joyal and Lurie? Recent work of \citet{Riehl-Verity:2013a,Riehl-Verity:2013b} has shown that questions of this nature can be answered by looking at what one might call the homotopy bicategory of the $\tuple{\infty, 2}$-category of $\tuple{\infty, 1}$-categories, at least when the latter is realised as the cartesian closed category of quasicategories.

We now briefly recall the aforementioned construction. Let $\cat{\SSet}$ be the category of simplicial sets, let $\cat{\Qcat}$ be the full subcategory of quasicategories, let $\nv : \cat{\Cat} \to \cat{\SSet}$ be the nerve functor, and let $\tau_1 : \cat{\SSet} \to \cat{\Cat}$ be its left adjoint. The Joyal model structure for quasicategories is cartesian (in the sense of \citet[\Sect 2]{Rezk:2010}), so $\cat{\Qcat}$ is an exponential ideal in $\cat{\SSet}$; in particular, $\cat{\Qcat}$ is cartesian closed. One can show that $\tau_1 : \cat{\SSet} \to \cat{\Cat}$ preserves finite products, so the cartesian closed structure of $\cat{\Qcat}$ gives rise to a 2-category $\bicat{\Qcat}$ whose underlying 1-category is (isomorphic to) $\cat{\Qcat}$. Moreover, a morphism in $\bicat{\Qcat}$ is an equivalence in the 2-categorical sense if and only if it is a weak equivalence in the Joyal model structure, so this is a good 2-truncation of the $\tuple{\infty, 2}$-category of $\tuple{\infty, 1}$-categories (in some sense still to be made precise).

Essentially the same story goes through for complete Segal spaces, though the proofs are (by necessity) more complicated. The purpose of the present paper is threefold:
\begin{itemize}
\item To give evidence for the correctness of the Joyal--Riehl--Verity construction of the homotopy bicategory of $\tuple{\infty, 1}$-categories, by showing that its bicategorical equivalence class is essentially theory-independent.

\item To show that an analogous construction can be carried out using complete Segal spaces instead of quasicategories.

\item To exhibit an explicit bicategorical equivalence between the two constructions.
\end{itemize}
In the first section, we will review some basic facts about cartesian closed categories and cartesian model structures. This will be followed by some observations about the 2-category $\bicat{\Qcat}$ and its intrinsicality with respect to the homotopy theory of $\tuple{\infty, 1}$-categories. We will then prove some propositions for Segal spaces needed to make the construction go through, and in the last section, we will show that one of the Quillen equivalences constructed by \citet{Joyal-Tierney:2007} induces the desired bicategorical equivalence.

\subsection*{Acknowledgements}

The author is indebted to Mike Shulman for providing the proof of \autoref{thm:bicategorical.equivalences}, which lies at the heart of the argument that the bicategorical equivalence class of the homotopy bicategory of the $\tuple{\infty, 2}$-category of $\tuple{\infty, 1}$-categories is well-defined. Thanks are also due to Emily Riehl for pointing out the work of \citet{Toen:2005}.

The author gratefully acknowledges financial support from the Cambridge Commonwealth, European and International Trust and the Department of Pure Mathematics and Mathematical Statistics.

\section{Generalities}

Recall the following definition:

\begin{dfn}
A \strong{cartesian closed category} is a category $\mathcal{C}$ equipped with the following data:
\begin{itemize}
\item A terminal object $1$.

\item A functor $\blank \times \blank : \mathcal{C} \times \mathcal{C} \to \mathcal{C}$ such that $X \times Y$ is the cartesian product of $X$ and $Y$ (in a natural way).

\item A functor $\xHom{\blank}{\blank} : \mathcal{C} \times \mathcal{C} \to \mathcal{C}$ equipped with bijections
\[
\Hom[\mathcal{C}]{X \times Y}{Z} \cong \Hom[\mathcal{C}]{X}{\xHom{Y}{Z}}
\]
that are natural in $X$, $Y$, and $Z$.
\end{itemize}
A \strong{cartesian closed functor} is a functor $F : \mathcal{C} \to \mathcal{D}$ between cartesian closed categories satisfying these conditions:
\begin{itemize}
\item $F$ preserves finite products (up to isomorphism).

\item The canonical morphism $F \xHom{X}{Y} \to \xHom{F X}{F Y}$ is an isomorphism for all $X$ and $Y$ in $\mathcal{C}$.
\end{itemize}
\end{dfn}

\begin{examples}
The following categories are cartesian closed:
\begin{enumerate}[(a)]
\item The category of sets, $\cat{\Set}$.

\item The category of small categories, $\cat{\Cat}$.

\item The category of small simplicially-enriched categories, $\cat{\SCat}$.

\item The category of presheaves on any small category.
\end{enumerate}
\end{examples}

It is not hard to check that a cartesian closed category $\mathcal{C}$ has a canonical $\mathcal{C}$-enrichment, with the object of morphisms $X \to Y$ being the exponential object $\xHom{X}{Y}$. As such, if $\mathcal{D}$ is a category with finite products and $F : \mathcal{C} \to \mathcal{D}$ is a functor that preserves finite products, there is an induced $\mathcal{D}$-enriched category $F \argb{\ul{\mathcal{C}}}$ with the same objects as $\mathcal{C}$ but where the object of morphisms $X \to Y$ is given by $F \xHom{X}{Y}$. (Of course, identities and composition are inherited from $\mathcal{C}$ via $F$.) Moreover:

\begin{lem}
\label{lem:transported.cartesian.closed.structure}
The underlying ordinary category of the $\mathcal{D}$-enriched category $F \argb{\ul{\mathcal{C}}}$ inherits the structure of a cartesian closed category from $\mathcal{C}$.
\end{lem}
\begin{proof}
It is easy to check that $\xHom{X}{1} \cong 1$ for all $X$ in $\mathcal{C}$, so $1$ is still a terminal object in $F \argb{\ul{\mathcal{C}}}$. Since $F$ preserves binary products, we have the following natural isomorphisms:
\[
F \xHom{X}{Y \times Z}
\cong F \argp{\xHom{X}{Y} \times \xHom{X}{Z}}
\cong F \xHom{X}{Y} \times F \xHom{X}{Z}
\]
It therefore follows that $Y \times Z$ is a cartesian product of $Y$ and $Z$ in $F \argb{\ul{\mathcal{C}}}$ as well. Finally, there is a natural isomorphism
\[
\xHom{X \times Y}{Z} \cong \xHom{X}{\xHom{Y}{Z}}
\]
and so by applying $F$ we deduce that $\xHom{Y}{Z}$, regarded as an object in $F \argb{\ul{\mathcal{C}}}$, is an exponential object of $Z$ by $Y$.
\end{proof}

\begin{dfn}
Let $\mathcal{C}$ be a cartesian closed category. A \strong{reflective exponential ideal} of $\mathcal{C}$ is a reflective subcategory $\mathcal{D}$ with the following property:
\begin{itemize}
\item For all $C$ in $\mathcal{C}$ and all $D$ in $\mathcal{D}$, the exponential object $\xHom{C}{D}$ (as computed in $\mathcal{C}$) is isomorphic to an object in $\mathcal{D}$.
\end{itemize}
\end{dfn}

\begin{prop}
\label{prop:reflective.exponential.ideals}
Let $\mathcal{C}$ be a cartesian closed category, let $G : \mathcal{D} \to \mathcal{C}$ be a fully faithful functor, and let $F : \mathcal{C} \to \mathcal{D}$ be a left adjoint of $G$. The following are equivalent:
\begin{enumerate}[(i)]
\item The image of $G : \mathcal{D} \to \mathcal{C}$ is a reflective exponential ideal of $\mathcal{C}$.

\item The functor $F : \mathcal{C} \to \mathcal{D}$ preserves finite products.

\item $\mathcal{D}$ is a cartesian closed category and, for all $C$ in $\mathcal{C}$ and all $D$ in $\mathcal{D}$, the canonical morphisms
\begin{align*}
G \xHom{F C}{D} & \to \xHom{G F C}{G D} &
\xHom{G F C}{G D} & \to \xHom{C}{G D}
\end{align*}
are isomorphisms.
\end{enumerate}
\end{prop}
\begin{proof} \openproof
See Proposition 4.3.1 in \citep[Part A]{Johnstone:2002a}.
\end{proof}

\begin{lem}
\label{lem:enrichment.via.transport}
Let $\mathcal{C}$ be a cartesian closed category and let $F : \mathcal{C} \to \mathcal{D}$ be a functor that preserves finite products. The functor $\mathcal{C} \to F \argb{\ul{\mathcal{C}}}$ induced by $F$ is an isomorphism of ordinary categories if and only if the hom-set maps
\[
\Hom[\mathcal{C}]{1}{X} \to \Hom[\mathcal{D}]{F 1}{F X}
\]
induced by $F$ are bijections for all objects $X$ in $\mathcal{C}$. 
\end{lem}
\begin{proof}
Simply recall the natural bijection $\Hom[\mathcal{C}]{X}{Y} \cong \Hom[\mathcal{C}]{1}{\xHom{X}{Y}}$ and the natural isomorphism $X \cong \xHom{1}{X}$.
\end{proof}

Throughout this paper, we use the definition of `model category' found in \citep[\Sect 1.1]{Hovey:1999}, \ie our model categories are complete, cocomplete, and have functorial factorisations. 

\begin{dfn}
\needspace{3.0\baselineskip}
A \strong{cartesian model structure} on a cartesian closed category is a model structure satisfying the following additional axioms:
\begin{itemize}
\item If $p : X \to Y$ is a fibration and $i : Z \to W$ is a cofibration, then in any commutative diagram
\[
\begin{tikzcd}
\xHom{W}{X} \arrow[swap, bend right=15]{ddr}{\xHom{W}{p}} \arrow[bend left=15]{drr}{\xHom{i}{X}} \drar[dashed]{q} \\
& 
L \argp{i, p} \dar \rar &
\xHom{Z}{X} \dar{\xHom{Z}{p}} \\
&
\xHom{W}{Y} \rar[swap]{\xHom{i}{Y}} &
\xHom{Z}{Y}
\end{tikzcd}
\]
where the square in the lower right is a pullback square, the morphism $q : \xHom{W}{X} \to L \argp{i, p}$ indicated in the diagram is a fibration, and $q$ is a weak equivalence if either $i$ or $p$ is a weak equivalence.

\item The terminal object $1$ is cofibrant.
\end{itemize}
A \strong{cartesian model category} is a cartesian closed category equipped with a model structure.
\end{dfn}

\begin{example}
The Kan--Quillen model structure on $\cat{\SSet}$ is a cartesian model structure: see \eg Proposition 4.2.8 in \citep{Hovey:1999}.
% Proposition 5.2 in \citep[\Chap I]{GJ}.
\end{example}

\begin{thm}
\needspace{3.0\baselineskip}
Let $\mathcal{M}$ be a cartesian model category, let $\mathcal{M}_\mathrm{f}$ be the full subcategory of fibrant objects, and let $\Ho \mathcal{M}$ be the homotopy category of $\mathcal{M}$. 
\begin{enumerate}[(i)]
\item $\mathcal{M}_\mathrm{f}$ is closed under small products in $\mathcal{M}$, and $\xHom{X}{Y}$ is fibrant if $X$ is cofibrant and $Y$ is fibrant.

\item The localisation functor $\gamma : \mathcal{M}_\mathrm{f} \to \Ho \mathcal{M}$ preserves small products; in particular, $\Ho \mathcal{M}$ has products for all small families of objects.

\item $\Ho \mathcal{M}$ is a cartesian closed category, and $\gamma \xHom{X}{Y}$ is naturally isomorphic to $\xHom{\gamma X}{\gamma Y}$ when $X$ is cofibrant and $Y$ is fibrant.

\item Let $\Gamma : \mathcal{M} \to \cat{\Set}$ be the functor $\Hom[\mathcal{M}]{1}{\blank}$ and let $\tau_0 : \mathcal{M} \to \cat{\Set}$ be the functor $\Hom[\Ho \mathcal{M}]{\gamma 1}{\gamma \blank}$. The functor $\tau_0$ preserves small products in $\mathcal{M}_\mathrm{f}$, and the component $\chi_Y : \Gamma Y \hoto \tau_0 Y$ of the natural transformation $\chi : \Gamma \hoto \tau_0$ induced by the functor $\gamma$ is surjective for all fibrant objects $Y$ in $\mathcal{M}$.
%
%\item Let $I$ be an object in $\mathcal{M}$, let $i_0, i_1 : 1 \to I$ be morphisms, and let $p$ be the unique morphism $I \to 1$. If $\tuple{I, i_0, i_1, p}$ is a cylinder object for $1$, then $\tuple{\xHom{I}{X}, i, p_0, p_1}$ is a path object for all fibrant  $X$, where $i : X \to \xHom{I}{X}$ is the morphism induced by $p : I \to 1$, and $p_0, p_1 : \xHom{I}{X} \to X$ are (respectively) the morphisms induced by $i_0, i_1 : 1 \to I$.
\end{enumerate}
\end{thm}
\begin{proof}
(i). That $\mathcal{M}_\mathrm{f}$ is closed in $\mathcal{M}$ under small products is a straightforward consequence of the right lifting property of fibrations, and the compatibility axiom for cartesian model structures implies the other half of the claim.

\bigskip\noindent
(ii). It can be shown that the canonical functor $\Ho \Func{J}{\mathcal{M}} \to \Func{J}{\Ho \mathcal{M}}$ is an equivalence of categories for all sets $J$, so products in $\Ho \mathcal{M}$ coincide with homotopy products. Homotopy products in $\mathcal{M}_\mathrm{f}$ coincide with ordinary products, hence the localisation functor $\gamma : \mathcal{M}_\mathrm{f} \to \Ho \mathcal{M}$ preserves small products. Since every object in $\mathcal{M}$ is weakly equivalent to one in $\mathcal{M}_\mathrm{f}$, it follows that $\Ho \mathcal{M}$ has products for all small families of objects.

\bigskip\noindent
(iii). See Theorem 4.3.2 in \citep{Hovey:1999}.

\bigskip\noindent
(iv). As a representable functor, $\Hom[\Ho \mathcal{M}]{\gamma 1}{\blank} : \Ho \mathcal{M} \to \cat{\Set}$ preserves small products, and by claim (ii), $\gamma : \mathcal{M}_\mathrm{f} \to \Ho \mathcal{M}$ preserves small products, so $\tau_0 : \mathcal{M}_\mathrm{f} \to \cat{\Set}$ indeed preserves small products. 

It is well-known that the localisation functor induces hom-set maps
\[
\Hom[\mathcal{M}]{X}{Y} \to \Hom[\Ho \mathcal{M}]{\gamma X}{\gamma Y}
\]
that are surjective when $X$ is cofibrant and $Y$ is fibrant; but $1$ is cofibrant by hypothesis, so the map $\chi_Y : \Gamma Y \to \tau_0 Y$ is surjective for all fibrant  $Y$.
%
%\bigskip\noindent
%(v). Since $1$ is a cofibrant object, $1 + 1$ is cofibrant, and hence $I$ itself is cofibrant. Ken Brown's lemma implies that the functor $\xHom{\blank}{X} : \op{\mathcal{M}} \to \mathcal{M}$ sends weak equivalences between cofibrant objects in $\mathcal{M}$ to weak equivalences between fibrant objects in $\mathcal{M}$ when $X$ is fibrant, so it follows that the morphism $i : X \to \xHom{I}{X}$ is a weak equivalence. Similarly, since the morphism $1 + 1 \to I$ induced by $i_0$ and $i_1$ is a cofibration, the morphism $\xHom{I}{X} \to X \times X$ induced by $p_0$ and $p_1$ is a fibration, so $\tuple{\xHom{I}{X}, i, p_0, p_1}$ is indeed a path object for $X$.
\end{proof}

Under stronger hypotheses, the homotopy category of a cartesian model category admits a description à la Hurewicz:

\begin{prop}
\label{prop:Hurewicz.homotopy.category}
Let $\mathcal{M}$ be a cartesian model category, let $\mathcal{M}_\mathrm{f}$ be the full subcategory of fibrant objects, and let $\Ho \mathcal{M}_\mathrm{f}$ be the localisation of $\mathcal{M}_\mathrm{f}$ at the weak equivalences. If all fibrant objects in $\mathcal{M}$ are cofibrant, then:
\begin{enumerate}[(i)]
\item $\mathcal{M}_\mathrm{f}$ is a cartesian closed category.

\item The natural transformation $\chi : \Gamma \hoto \tau_0$ induces a functor $\mathcal{M}_\mathrm{f} \to \tau_0 \argb{\ul{\mathcal{M}_\mathrm{f}}}$ that is a bijection on objects, full, and preserves small products and exponential objects. 
%
%\item Let $f_0, f_1 : X \to Y$ be a parallel pair of morphisms in $\mathcal{M}_\mathrm{f}$. Then $f_0$ and $f_1$ are (right) homotopic if and only if they are sent to the same morphism in $\tau_0 \argb{\ul{\mathcal{M}_\mathrm{f}}}$.

\item The canonical functor $\Ho \mathcal{M}_\mathrm{f} \to \tau_0 \argb{\ul{\mathcal{M}_\mathrm{f}}}$ is an isomorphism of categories.
\end{enumerate}
\end{prop}
\begin{proof}
(i). Since all fibrant objects are cofibrant, the exponential object $\xHom{X}{Y}$ is fibrant for all $X$ and $Y$ in $\mathcal{M}_\mathrm{f}$; since $\mathcal{M}_\mathrm{f}$ is closed under small products in $\mathcal{M}$, it follows that $\mathcal{M}_\mathrm{f}$ is a cartesian closed category.

\bigskip\noindent
(ii). This is a straightforward consequence of the fact that $\tau_0 : \mathcal{M}_\mathrm{f} \to \cat{\Set}$ preserves small products, that we have a natural bijection $\Gamma \xHom{X}{Y} \cong \Hom[\mathcal{M}]{X}{Y}$ for all objects $X$ and $Y$, and that $\chi_Z : \Gamma Z \to \tau_0 Z$ is a surjection for all fibrant objects $Z$.
%
%\bigskip\noindent
%(iii). Suppose $f_0, f_1 : X \to Y$ are related by a right homotopy, \ie there exists a path object $\tuple{P, i, p_0, p_1}$ for $Y$ and a morphism $f : X \to P$ such that $p_0 \circ f = f_0$ and $p_1 \circ f = f_1$. Since $p_0, p_1 : P \to Y$ are retractions of the weak equivalence $i : Y \to P$, the two maps $\tau_0 \xHom{X}{P} \to \tau_0 \xHom{X}{Y}$ induced by $p_0$ and $p_1$ must be equal. In particular, $\chi_{\xHom{X}{Y}} : \Gamma \xHom{X}{Y} \to \tau_0 \xHom{X}{Y}$ must map $f_0$ and $f_1$ to the same element.
%
%Conversely, if $f_0$ and $f_1$ are sent to the same morphism in $\tau_0 \argb{\ul{\mathcal{M}_\mathrm{f}}}$, then there must exist a cylinder object $\tuple{I, i_0, i_1, p}$ for $1$ and a morphism $h : I \to \xHom{X}{Y}$ such that $h \circ i_0$ (\resp $h \circ i_1$) is the exponential transpose of $f_0$ (\resp $f_1$). Taking exponential transposes again and using the fact that $\xHom{I}{Y}$ is a path object for $Y$, we deduce that $f_0$ and $f_1$ are right homotopic.
%
%\bigskip\noindent
%(iv). The formal Whitehead theorem implies that weak equivalences in $\mathcal{M}_\mathrm{f}$ are mapped to isomorphisms in $\tau_0 \argb{\ul{\mathcal{M}_\mathrm{f}}}$, so the functor $\mathcal{M} \to \tau_0 \argb{\ul{\mathcal{M}_\mathrm{f}}}$ induces a functor $\Ho \mathcal{M}_\mathrm{f} \to \tau_0 \argb{\ul{\mathcal{M}_\mathrm{f}}}$. A standard argument then shows that it is an isomorphism: see \eg Proposition 1.2.5 in \citep{Hovey:1999}.

\bigskip\noindent
(iii). By definition, $\tau_0 \argb{\ul{\mathcal{M}_\mathrm{f}}}$ is the underlying category of the $\parens{\Ho \mathcal{M}_\mathrm{f}}$-enriched category $\ul{\Ho \mathcal{M}_\mathrm{f}}$, so it is canonically isomorphic to $\Ho \mathcal{M}_\mathrm{f}$.
\end{proof}

\begin{prop}
\label{prop:Quillen.homotopy.category.as.a.reflective.subcategory}
\needspace{3.0\baselineskip}
Let $\mathcal{M}$ be a cartesian model category. If all objects in $\mathcal{M}$ are cofibrant, then:
\begin{enumerate}[(i)]
\item The functors $\gamma : \mathcal{M} \to \Ho \mathcal{M}$ and $\tau_0 : \mathcal{M} \to \cat{\Set}$ both preserve finite products.

\item A morphism $f : X \to Y$ in $\mathcal{M}$ is a weak equivalence if and only if the induced maps
\[
\tau_0 \xHom{f}{Z} : \tau_0 \xHom{Y}{Z} \to \tau_0 \xHom{X}{Z}
\]
are bijections for all fibrant objects $Z$ in $\mathcal{M}$.

\item The inclusion $\mathcal{M}_\mathrm{f} \embedinto \mathcal{M}$ induces a fully faithful functor $\tau_0 \argb{\ul{\mathcal{M}}_\mathrm{f}} \to \tau_0 \argb{\ul{\mathcal{M}}}$ with a left adjoint.
\end{enumerate}
\end{prop}
\begin{proof}
(i). It suffices to show that $\gamma : \mathcal{M} \to \Ho \mathcal{M}$ preserves finite products; that $\tau_0 : \mathcal{M} \to \cat{\Set}$ preserves finite products will follow automatically. It is not hard to check that $\gamma : \mathcal{M} \to \Ho \mathcal{M}$ preserves terminal objects for all model categories $\mathcal{M}$, and we will now show that $\gamma$ preserves binary products. 

The definition of cartesian model category implies that, for all cofibrant objects $Y$, the functor $\blank \times Y : \mathcal{M} \to \mathcal{M}$ is a left Quillen functor. Since we are assuming that all objects are cofibrant, Ken Brown's lemma implies that $\blank \times Y$ preserves weak equivalences. We then deduce that $\blank \times \blank : \mathcal{M} \times \mathcal{M} \to \mathcal{M}$ preserves all weak equivalences, and hence that it is its own left derived functor. Thus, the localisation functor $\gamma : \mathcal{M} \to \Ho \mathcal{M}$ indeed preserves binary products.

\bigskip\noindent
(ii). If $f : X \to Y$ is a weak equivalence, then $\xHom{f}{Z} : \xHom{Y}{Z} \to \xHom{X}{Z}$ is a weak equivalence for all fibrant objects $Z$, and hence $\tau_0 \xHom{f}{Z}$ must be a bijection. Conversely, suppose $\tau_0 \xHom{f}{Z}$ is a bijection for all fibrant objects $Z$. Let $R : \mathcal{M} \to \mathcal{M}$ be a fibrant replacement functor for $\mathcal{M}$. Then, the morphism $R f : R X \to R Y$ also induces bijections $\tau_0 \xHom{R f}{Z}$ for all fibrant objects $Z$, and since $R X$ and $R Y$ are in $\mathcal{M}_\mathrm{f}$, the Yoneda lemma implies that $R f : R X \to R Y$ is sent to an isomorphism in $\tau_0 \argb{\ul{\mathcal{M}}_\mathrm{f}}$, and hence must be a weak equivalence in $\mathcal{M}_\mathrm{f}$. The 2-out-of-3 property of weak equivalences then implies $f : X \to Y$ is a weak equivalence in $\mathcal{M}$.

\bigskip\noindent
(iii). It is clear that the induced functor $\tau_0 \argb{\ul{\mathcal{M}}_\mathrm{f}} \to \tau_0 \argb{\ul{\mathcal{M}}}$ is indeed fully faithful, and it is not hard to check that a fibrant replacement functor provides the required left adjoint $\tau_0 \argb{\ul{\mathcal{M}}} \to \tau_0 \argb{\ul{\mathcal{M}}_\mathrm{f}}$.
\end{proof}

\section{Quasicategories}

We begin with a bit of category-theoretic folklore.

\begin{thm}
\needspace{2.5\baselineskip}
The following data define a cartesian model structure on $\cat{\Cat}$:
\begin{itemize}
\item The weak equivalences are the \strong{categorical equivalences}, \ie functors that are fully faithful and essentially surjective on objects.

\item The cofibrations are the \strong{isocofibrations}, \ie functors that are injective on objects.

\item The fibrations are the \strong{isofibrations}, \ie functors $F : \mathcal{C} \to \mathcal{D}$ with the isomorphism lifting property, \ie for all objects $C$ in $\mathcal{C}$ and all isomorphisms $f : F C \to D$ in $\mathcal{D}$, there exists an isomorphism $\tilde{f} : C \to C'$ in $\mathcal{C}$ such that $F C' = D$ and $F \tilde{f} = f$.
\end{itemize}
We refer to this as the \strong{natural model structure} on $\cat{\Cat}$.
\end{thm}
\begin{proof} \openproof
See \citep{Rezk:1996}.
\end{proof}

\begin{remark}
The natural model structure on $\cat{\Cat}$ makes all objects cofibrant and fibrant, and two functors are homotopic if and only if they are naturally isomorphic.
\end{remark}

To fix notation and terminology, we recall a few definitions and results:

\begin{dfn}
A \strong{quasicategory} is a simplicial set $X$ that satisfies the Boardman condition: for all $n \ge 2$ and all $0 < i < n$, the unique morphism $X \to 1$ has the right lifting property with respect to the (inner) horn inclusion $\Lambda^n_i \embedinto \Delta^n$.
\end{dfn}

\begin{prop}
Let $\nv : \cat{\Cat} \to \cat{\SSet}$ be the functor that sends a small category to its nerve.
\begin{enumerate}[(i)]
\item The functor $\nv$ is fully faithful and preserves colimits for all small filtered diagrams and limits for all diagrams.

\item $\nv$ has a left adjoint, $\tau_1 : \cat{\SSet} \to \cat{\Cat}$.

\item The functor $\tau_1$ preserves finite products.
\end{enumerate}
\end{prop}
\begin{proof}
(i). This is well-known and not hard to check.

\bigskip\noindent
(ii). Apply the accessible adjoint functor theorem, or use the theory of Kan extensions.

\bigskip\noindent
(iii). Since both $\cat{\Cat}$ and $\cat{\SSet}$ are cartesian closed, $\blank \times \blank$ preserves colimits in both variables in both categories. It therefore suffices to check that the canonical morphism $\tau_1 \parens{\Delta^n \times \Delta^m} \to \tau_1 \Delta^n \times \tau_1 \Delta^m$ is an isomorphism for all pairs of natural numbers $n$ and $m$; but this is easy because $\nv : \cat{\Cat} \to \cat{\SSet}$ is fully faithful and each $\Delta^n$ is the nerve of a small category.
\end{proof}
%
%\begin{cor}
%The adjunction displayed below admits an $\cat{\SSet}$-enrichment and a $\cat{\Cat}$-enrichment,
%\[
%\tau_1 \dashv \nv : \cat{\Cat} \to \cat{\SSet}
%\]
%where the $\cat{\SSet}$-enrichment of $\cat{\Cat}$ is induced by $\nv : \cat{\Cat} \to \cat{\SSet}$ and the $\cat{\Cat}$-enrichment of $\cat{\SSet}$ is induced by $\tau_1 : \cat{\SSet} \to \cat{\Cat}$.
%\end{cor}
%\begin{proof}
%Since $\tau_1$ preserves binary products, we have the following natural bijections:
%\begin{align*}
%\Hom[\SSet]{X}{\nv{\Func{\tau_1 Y}{\mathcal{B}}}}
%& \cong \Hom[\Cat]{\tau_1 X}{\Func{\tau_1 Y}{\mathcal{B}}} \\
%& \cong \Hom[\Cat]{\tau_1 X \times \tau_1 Y}{\mathcal{B}} \\
%& \cong \Hom[\Cat]{\tau_1 \argp{X \times Y}}{\mathcal{B}} \\
%& \cong \Hom[\SSet]{X \times Y}{\nv{\mathcal{B}}} \\
%& \cong \Hom[\SSet]{X}{\xHom{Y}{\nv{\mathcal{B}}}}
%\end{align*}
%Thus, by the Yoneda lemma, we have a natural isomorphism
%\[
%\nv{\Func{\tau_1 Y}{\mathcal{B}}} \cong \xHom{Y}{\nv{\mathcal{B}}}
%\]
%and thus we see that [...]
%
%and since the adjunction counit $\tau_1 \nv \hoto \id_{\cat{\Cat}}$ is a natural isomorphism, we deduce that
%\[
%\Func{\tau_1 Y}{\mathcal{B}} \cong \tau_1 \xHom{Y}{\nv{\mathcal{B}}}
%\]
%as well. 
%\end{proof}

\begin{remark}
There is a well-known description of $\tau_1 X$ when $X$ is a quasicategory:
\begin{itemize}
\item The objects are the vertices of $X$.

\item The morphisms $x_0 \to x_1$ are the edges $f$ such that $d_1 \argp{f} = x_0$ and $d_0 \argp{f} = x_1$, modulo the equivalence relation $f_0 \sim f_1$ generated by 2-simplices $\sigma$ such that $d_2 \argp{\sigma} = f_0$, $d_1 \argp{\sigma} = f_1$, and $d_0 \argp{\sigma} = s_0 \argp{x_1}$.

\item The identity morphism $x \to x$ is (the equivalence class of) the edge $s_0 \argp{x}$.

\item Composition is defined by 2-simplices: if $\sigma$ is a 2-simplex of $X$ with $d_2 \argp{\sigma} = f_1$ and $d_0 \argp{\sigma} = f_2$, then $f_2 \circ f_1$ is defined to be $d_1 \argp{\sigma}$.
\end{itemize}
Moreover, $X$ is a Kan complex if and only if $X$ is a quasicategory and $\tau_1 X$ is a groupoid. For details, see Propositions 1.2.3.9 and 1.2.5.1 in \citep{HTT}. 
\end{remark}

Let $\tau_0 : \cat{\SSet} \to \cat{\Set}$ be the functor that sends a simplicial set $X$ to the set of isomorphism classes of objects in the category $\tau_1 X$.

\begin{dfn}
%The \strong{fundamental category} of a simplicial set $X$ is the category $\tau_1 X$.
A \strong{weak categorical equivalence} is a morphism $F : X \to Y$ in $\cat{\SSet}$ such that the induced maps
\[
\tau_0 \xHom{F}{Z} : \tau_0 \xHom{Y}{Z} \to \tau_0 \xHom{X}{Z}
\]
are bijections for all quasicategories $Z$.
\end{dfn}

\begin{thm}[Joyal]
\needspace{2.5\baselineskip}
The following data define a cartesian model structure on $\cat{\SSet}$:
\begin{itemize}
\item The weak equivalences are the weak categorical equivalences.

\item The cofibrations are the monomorphisms.
\end{itemize}
We refer to this as the \strong{Joyal model structure for quasicategories}, and the fibrant objects in this model structure are the quasicategories.
\end{thm}
\begin{proof} \openproof
See Theorem 6.12 in \citep{Joyal:2008}.
\end{proof}

\begin{prop}
\label{prop:truncation-nerve.adjunction.for.quasicategories}
The adjunction displayed below
\[
\tau_1 \dashv \nv : \cat{\Cat} \to \cat{\SSet}
\]
is a Quillen adjunction with respect to the natural model structure on $\cat{\Cat}$ and the Joyal model structure on $\cat{\SSet}$.
\end{prop}
\begin{proof} \openproof
See Proposition 6.14 in \citep{Joyal:2008}.
\end{proof}

\begin{cor}
\needspace{3.0\baselineskip}
Let $\cat{\Qcat}$ be the full subcategory of $\cat{\SSet}$ spanned by the quasicategories.
\begin{enumerate}[(i)]
\item $\cat{\Qcat}$ is closed under small products and exponential objects in $\cat{\SSet}$; in particular, $\cat{\Qcat}$ is a cartesian closed category.

\item Let $\Ho \cat{\Qcat}$ be the localisation of $\cat{\Qcat}$ at the weak categorical equivalences. The localisation functor $\cat{\Qcat} \to \Ho \cat{\Qcat}$ preserves small products and exponential objects.

\item If $X$ and $Y$ are quasicategories, then the categorical equivalence class of $\tau_1 \xHom{X}{Y}$ depends only on the weak categorical equivalence classes of $X$ and $Y$. 
\end{enumerate}
\end{cor}
\begin{proof}
Apply \autoref{prop:Hurewicz.homotopy.category}.
\end{proof}

\begin{lem}
\label{lem:weak.categorical.equivalences}
Let $F : X \to Y$ be a morphism in $\cat{\SSet}$. The following are equivalent:
\begin{enumerate}[(i)]
\item For all quasicategories $Z$, the induced morphisms
\[
\xHom{F}{Z} : \xHom{Y}{Z} \to \xHom{X}{Z}
\]
are weak categorical equivalences.

\item For all quasicategories $Z$, the induced functors
\[
\tau_1 \xHom{F}{Z} : \xHom{Y}{Z} \to \xHom{X}{Z}
\]
are categorical equivalences.

\item $F : X \to Y$ is a weak categorical equivalence.
\end{enumerate}
\end{lem}
\begin{proof}
By \autoref{prop:Quillen.homotopy.category.as.a.reflective.subcategory}, this is a formal consequence of the fact that the Joyal model structure for quasicategories is cartesian and that $\tau_1 : \cat{\SSet} \to \cat{\Cat}$ preserves weak equivalences.
\end{proof}

\begin{remark}
Let $\bicat{\Qcat}$ be the 2-category defined by $\tau_1 \argb{\ul{\cat{\Qcat}}}$. \Autoref{lem:weak.categorical.equivalences} says that a weak categorical equivalence between quasicategories is the same thing as an equivalence in $\bicat{\Qcat}$, and \autoref{prop:Hurewicz.homotopy.category} implies that a parallel pair of morphisms in $\cat{\Qcat}$ become equal in $\Ho \cat{\Qcat}$ if and only if they are isomorphic in $\bicat{\Qcat}$. Thus, we may think of $\bicat{\Qcat}$ as being (a model for) the homotopy bicategory of an $\tuple{\infty, 2}$-category of quasicategories.
\end{remark}

\begin{prop}
\needspace{3.0\baselineskip}
Let $\Ho \cat{\Cat}$ be the localisation of $\cat{\Cat}$ at the categorical equivalences. 
\begin{enumerate}[(i)]
\item The adjunction $\tau_1 \dashv \nv : \cat{\Cat} \to \cat{\Qcat}$ descends to an adjunction of homotopy categories,
\[
\Ho \tau_1 \dashv \Ho \nv : \Ho \cat{\Cat} \to \Ho \cat{\Qcat}
\]
and the functor $\Ho \nv$ is fully faithful.

\item The functor $\Ho \tau_1 : \Ho \cat{\Qcat} \to \Ho \cat{\Cat}$ preserves finite products.

\item $\Ho \tau_1$ induces a $\parens{\Ho \cat{\Cat}}$-enrichment of $\Ho \cat{\Qcat}$.
\end{enumerate}
\end{prop}
\begin{proof}
(i). Since all objects in $\cat{\Cat}$ are fibrant, Ken Brown's lemma implies that $\nv$ sends categorical equivalences in $\cat{\Cat}$ to weak categorical equivalences in $\cat{\Qcat}$; and since all objects in $\cat{\SSet}$ are cofibrant in the Joyal model structure, $\tau_1$ sends weak categorical equivalences in $\cat{\Qcat}$ to categorical equivalences in $\cat{\Cat}$. It follows that there is a well-defined adjunction $\Ho \tau_1 \dashv \Ho \nv$. The counit of this adjunction is a natural isomorphism because the counit of the adjunction $\tau_1 \dashv \nv$ is a natural isomorphism, so we may deduce that $\Ho \nv$ is fully faithful.

\bigskip\noindent
(ii). Consider the following commutative diagram of functors:
\[
\begin{tikzcd}
\cat{\Qcat} \dar[swap]{\tau_1} \rar &
\Ho \cat{\Qcat} \dar{\Ho \tau_1} \\
\cat{\Cat} \rar &
\Ho \cat{\Cat}
\end{tikzcd}
\]
The horizontal arrows are functors that are bijective on objects and preserve finite products, and $\tau_1 : \cat{\Qcat} \to \cat{\Cat}$ preserves finite products, so we deduce that $\Ho \tau_1 : \Ho \cat{\Qcat} \to \Ho \cat{\Cat}$ does as well.

\bigskip\noindent
(iii). This is an immediate consequence of the fact that $\Ho \cat{\Qcat}$ is a cartesian closed category and that $\Ho \tau_1$ preserves finite products.
\end{proof}

A $\parens{\Ho \cat{\Cat}}$-enriched category can be thought of as something like a bicategory without coherence data; indeed, any bicategory gives rise to a $\parens{\Ho \cat{\Cat}}$-enriched category in an obvious way. Moreover:

\begin{thm}
\label{thm:bicategorical.equivalences}
Let $\mathcal{C}$ and $\mathcal{D}$ be two $\parens{\Ho \cat{\Cat}}$-enriched categories, and let $F : \mathcal{C} \to \mathcal{D}$ be a $\parens{\Ho \cat{\Cat}}$-enriched equivalence.
\begin{enumerate}[(i)]
\item If $F : \mathcal{C} \to \mathcal{D}$ underlies a pseudofunctor between bicategories, then that pseudofunctor is a bicategorical equivalence.

\item If $\mathcal{D}$ underlies a bicategory $\mathfrak{D}$, then $\mathcal{C}$ underlies a bicategory $\mathfrak{C}$ such that $F : \mathcal{C} \to \mathcal{D}$ underlies a bicategorical equivalence $\mathfrak{C} \to \mathfrak{D}$.
\end{enumerate}
\end{thm}
\begin{proof}
(i). This is just the definition of bicategorical equivalence.

\bigskip\noindent
(ii). By factorising $F : \mathcal{C} \to \mathcal{D}$, we can reduce the problem to the following two cases:
\begin{enumerate}[(a)]
\item $F$ is an identity-on-objects $\parens{\Ho \cat{\Cat}}$-enriched equivalence.

\item $F$ is a $\parens{\Ho \cat{\Cat}}$-enriched equivalence that acts as the identity on hom-objects.
\end{enumerate}
Case (b) is straightforward, so we focus on case (a).

To prove the claim, we will use some techniques from 2-dimensional category theory. We assume for simplicity that $\mathfrak{D}$ is a small bicategory. Recall that a $\cat{\Cat}$-valued graph is a set of objects together with a small category for each ordered pair of objects. Clearly, every small bicategory has an underlying $\cat{\Cat}$-valued graph, and there is an obvious 2-category of $\cat{\Cat}$-valued graphs over a fixed set of objects, say $O$, namely the 2-category $\bicat{\Cat}^{O \times O}$. The following facts are well-known:
\begin{itemize}
\item The 2-category of small 2-categories with object-set $O$ (whose morphisms are the identity-on-objects 2-functors) is 2-monadic over $\bicat{\Cat}^{O \times O}$.

\item Pseudoalgebras for the induced 2-monad on $\bicat{\Cat}^{O \times O}$ are unbiased small bicategories with object-set $O$, and the strong morphisms are identity-on-objects pseudofunctors.

\item Pseudoalgebra structures for any 2-monad can be transported along equivalences.
\end{itemize}
Now, let $O = \ob \mathcal{C}$, let $\mathfrak{C}$ be the $\cat{\Cat}$-valued graph determined by $\mathcal{C}$ and choose, for each pair of objects $X$ and $Y$ in $\mathcal{C}$, a categorical equivalence to represent the hom-object morphism $\Hom[\mathcal{C}]{X}{Y} \to \Hom[\mathcal{D}]{F X}{F Y}$. This defines an equivalence $\tilde{F} : \mathfrak{C} \to \mathfrak{D}$ in the 2-category $\bicat{\Cat}^{O \times O}$, so using the facts recalled above, we deduce that $\mathfrak{C}$ admits the structure of a bicategory making $\tilde{F} : \mathfrak{C} \to \mathfrak{D}$ into a bicategorical equivalence. By examining the construction of the transported bicategory structure on $\mathfrak{C}$, one may then deduce that $F : \mathcal{C} \to \mathcal{D}$ is indeed the $\parens{\Ho \cat{\Cat}}$-enriched functor underlying $\tilde{F} : \mathfrak{C} \to \mathfrak{D}$.
%\bigskip\noindent
%(ii). We will assume that $\mathcal{D}$ underlies a bicategory $\mathfrak{D}$; the other case is similar. We construct a bicategory $\mathfrak{C}$ as follows:
%\begin{itemize}
%\item The objects of $\mathfrak{C}$ are the objects of $\mathcal{C}$.
%
%\item For objects $X, Y$ in $\mathfrak{C}$, the hom-category $\HHom[\mathfrak{C}]{X}{Y}$ is (isomorphic to) the hom-category $\HHom[\mathfrak{D}]{F X}{F Y}$.
%
%\item Identities, composition, unitors, and associators are inherited from $\mathfrak{D}$.
%\end{itemize}
%It is clear that $\mathfrak{C}$ satisfies the coherence axioms for bicategories, simply because $\mathfrak{D}$ does. It is also easy to see that there is a bicategorical equivalence $\mathfrak{C} \to \mathfrak{D}$ whose underlying $\parens{\Ho \cat{\Cat}}$-enriched functor is isomorphic to $F : \mathcal{C} \to \mathcal{D}$.
\end{proof}

\begin{remark}
\label{rem:intrinsicality.of.the.homotopy.bicategory}
It is clear from the construction of the 2-category $\bicat{\Qcat}$ that the $\parens{\Ho \cat{\Cat}}$-enriched version of $\Ho \cat{\Qcat}$ is the $\parens{\Ho \cat{\Cat}}$-enriched category underlying $\bicat{\Qcat}$. Thus, if we accept the following hypotheses,
\begin{itemize}
\item $\Ho \cat{\Qcat}$ is the correct homotopy category of the $\tuple{\infty, 1}$-category of small $\tuple{\infty, 1}$-categories.

\item $\Ho \nv : \Ho \cat{\Cat} \to \Ho \cat{\Qcat}$ is the correct embedding of $\Ho \cat{\Cat}$ into $\Ho \cat{\Qcat}$.
\end{itemize}
then we should also believe that the $\parens{\Ho \cat{\Cat}}$-enriched incarnation of $\Ho \cat{\Qcat}$ underlies the correct homotopy bicategory of the $\tuple{\infty, 2}$-category of small $\tuple{\infty, 1}$-categories\hairspace ---\hairspace however we choose to make those notions precise. In particular, \autoref{thm:bicategorical.equivalences} implies that any such homotopy bicategory must be equivalent to $\bicat{\Qcat}$ as a bicategory.
\end{remark}
%
%\begin{remark}
%Unfortunately, neither $\tau_1 \argb{\ul{\cat{\SSet}}}$ nor $\bicat{\Qcat}$ are model 2-categories (in the sense of having a model structure compatible with the model structure on $\cat{\Cat}$): this is because they fail to have $\cat{\Cat}$-weighted colimits.
%\end{remark}

We finish off this section with a ``transport of structure'' construction that puts a 2-category structure on the full subcategory of cofibrant--fibrant objects in any model category Quillen-equivalent to the Joyal model category.

\begin{lem}
\label{lem:homotopical.equivalences.from.Quillen.equivalences}
Let $\mathcal{M}$ and $\mathcal{N}$ be model categories, and let $\mathcal{M}_\mathrm{cf}$ and $\mathcal{N}_\mathrm{cf}$ be the respective full subcategories of cofibrant--fibrant objects. If $\mathcal{M}$ and $\mathcal{N}$ are Quillen-equivalent, then there exists a functor $F : \mathcal{M}_\mathrm{cf} \to \mathcal{N}_\mathrm{cf}$ with the following properties:
\begin{itemize}
\item $F : \mathcal{M}_\mathrm{cf} \to \mathcal{N}_\mathrm{cf}$ preserves weak equivalences.

\item The induced functor $\Ho F : \Ho \mathcal{M}_\mathrm{cf} \to \Ho \mathcal{N}_\mathrm{cf}$ is a categorical equivalence.
\end{itemize}
\end{lem}
\begin{proof}
By induction on the length of the zigzag of Quillen equivalences connecting $\mathcal{M}$ and $\mathcal{N}$, we may reduce the claim to the case where there is either a left or right Quillen equivalence $\mathcal{M} \to \mathcal{N}$; by duality, we may assume without loss of generality that there is a right Quillen equivalence $R : \mathcal{M} \to \mathcal{N}$. 

For each object $N$ in $\mathcal{N}$, choose a functorial fibrant cofibrant replacement $p_N : Q N \to N$. Since $R$ preserves fibrant objects, $Q R M$ is a cofibrant--fibrant object in $\mathcal{N}$ for all fibrant objects $M$ in $\mathcal{M}$. We may then define the functor $F : \mathcal{M}_\mathrm{cf} \to \mathcal{N}_\mathrm{cf}$ by taking $F = Q R$. Ken Brown's lemma says that $R$ preserves weak equivalences between fibrant objects, and $p : Q \hoto \id_\mathcal{N}$ is a natural weak equivalence, so $F : \mathcal{M}_\mathrm{cf} \to \mathcal{N}_\mathrm{cf}$ must preserve all weak equivalences. There is then an induced functor $\Ho F : \Ho \mathcal{M}_\mathrm{cf} \to \Ho \mathcal{N}_\mathrm{cf}$, and it is a categorical equivalence because $R : \mathcal{M} \to \mathcal{N}$ is a right Quillen equivalence.
\end{proof}

\begin{prop}
\needspace{3.0\baselineskip}
Let $\mathcal{C}$ be a category with weak equivalences satisfying the following conditions:
\begin{itemize}
\item A morphism in $\mathcal{C}$ is a weak equivalence if and only if the localisation functor $\mathcal{C} \to \Ho \mathcal{C}$ sends it to an isomorphism.

\item The localisation functor $\mathcal{C} \to \Ho \mathcal{C}$ is full.
\end{itemize}
Let $\mathfrak{D}$ be a 2-category, let $\mathcal{D}$ be its underlying ordinary category, and let $\Ho \mathcal{D}$ be $\tau_0 \argb{\mathfrak{D}}$, \ie the category obtained from $\mathcal{D}$ by identifying parallel pairs of morphisms that are isomorphic. Suppose $F : \mathcal{C} \to \mathcal{D}$ is a functor that sends weak equivalences in $\mathcal{C}$ to equivalences in $\mathfrak{D}$. If the induced functor $\Ho F : \Ho \mathcal{C} \to \Ho \mathcal{D}$ is a categorical equivalence, then:
\begin{enumerate}[(i)]
\item There exist a unique (up to unique isomorphism) 2-category $\mathfrak{C}$ equipped with a bicategorical equivalence $\tilde{F} : \mathfrak{C} \to \mathfrak{D}$ whose underlying functor is $U : \mathcal{C} \to \mathcal{D}$.

\item A morphism in $\mathcal{C}$ is a weak equivalence if and only if it is an equivalence in the 2-category $\mathfrak{C}$.

\item A parallel pair of morphisms in $\mathcal{C}$ are sent to the same morphism in $\Ho \mathcal{C}$ if and only if they are isomorphic in $\mathfrak{C}$.
\end{enumerate}
\end{prop}
\begin{proof}
(i). The key observation is that, given any category $\mathcal{A}$ and any map $u : X \to \ob \mathcal{A}$ such that the composite $X \stackrel{u}{\to} \ob \mathcal{A} \to \tau_0 \mathcal{A}$ is surjective, there is a unique (up to unique isomorphism) category $\tilde{\mathcal{A}}$ equipped with a categorical equivalence $\tilde{u} : \tilde{\mathcal{A}} \to \mathcal{A}$ such that $\ob \tilde{u} : \ob \tilde{\mathcal{A}} \to \mathcal{A}$ is the map $u : X \to \ob \mathcal{A}$. It is now clear how to construct the 2-category $\mathfrak{C}$: given a parallel pair $f_0, f_1 : A \to B$ in $\mathcal{C}$, we define the set of natural transformations $f_0 \hoto f_1$ in $\mathfrak{C}$ to be the set of natural transformations $F f_0 \hoto F f_1$ in $\mathfrak{D}$, and we define horizontal and vertical composition by transport of structure. There is then an obvious bicategorical equivalence $\tilde{F} : \mathfrak{C} \to \mathfrak{D}$ whose underlying functor is $F : \mathcal{C} \to \mathcal{D}$.

\bigskip\noindent
(ii). Let $f : A \to B$ be a morphism in $\mathcal{C}$. First, suppose $f$ is a weak equivalence in $\mathcal{C}$. Then, $F f$ is an equivalence in $\mathfrak{D}$. Since $\tilde{F} : \mathfrak{C} \to \mathfrak{D}$ is a bicategorical equivalence, we deduce that $f$ is an equivalence in the 2-category $\mathfrak{C}$.

Conversely, suppose $f$ is an equivalence in the 2-category $\mathfrak{C}$. Then, $F f$ is an equivalence in the 2-category $\mathfrak{D}$, and hence is an isomorphism in $\Ho \mathcal{D}$. Since $\Ho F : \Ho \mathcal{C} \to \Ho \mathcal{D}$ is a categorical equivalence, $f$ must be an isomorphism in $\Ho \mathcal{C}$, so we deduce that $f$ is a weak equivalence in $\mathcal{C}$.

\bigskip\noindent
(iii). The proof is similar to that of claim (ii).
\end{proof}

\begin{cor}
Let $\mathcal{M}$ be a model category and let $\mathcal{M}_\mathrm{cf}$ be the full subcategory of cofibrant--fibrant objects in $\mathcal{M}$. If $\mathcal{M}$ is Quillen-equivalent to the Joyal model category for quasicategories, then there exist a 2-category $\mathfrak{M}$ and a 2-functor $\tilde{F} : \mathfrak{M} \to \bicat{\Qcat}$ with the following properties:
\begin{itemize}
\item $\tilde{F} : \mathfrak{M} \to \bicat{\Qcat}$ is a bicategorical equivalence.

\item The underlying functor of $\tilde{F}$ is a functor $F : \mathcal{M}_\mathrm{cf} \to \cat{\Qcat}$ that preserves weak equivalences.

\item The induced functor $\Ho F : \Ho \mathcal{M}_\mathrm{cf} \to \Ho \cat{\Qcat}$ is a categorical equivalence.
\end{itemize}
\end{cor}
\begin{proof}
Using \autoref{lem:homotopical.equivalences.from.Quillen.equivalences} and the well-known fact that the localisation functor $\mathcal{M}_\mathrm{cf} \to \Ho \mathcal{M}_\mathrm{cf}$ is full, we see that there is a functor $F : \mathcal{M} \to \cat{\Qcat}$ satisfying the hypotheses of the earlier proposition.
\end{proof}

\begin{remark}
As explained in \autoref{rem:intrinsicality.of.the.homotopy.bicategory}, the $\parens{\Ho \cat{\Cat}}$-enriched category underlying $\mathfrak{M}$ can be described intrinsically in terms of $\Ho \mathcal{M}$, at least once we know how $\Ho \cat{\Cat}$ is embedded in $\Ho \mathcal{M}$. We are able to obtain a straightforward description of the 2-category $\mathfrak{M}$ itself in this case because \autoref{lem:homotopical.equivalences.from.Quillen.equivalences} gives us the extra data of a functor $\mathcal{M}_\mathrm{cf} \to \cat{\Qcat}$ that descends to a categorical equivalence $\Ho \mathcal{M}_\mathrm{cf} \to \Ho \cat{\Qcat}$. 

It seems unlikely that $\mathfrak{M}$ can be described purely in terms of the $\parens{\Ho \cat{\Cat}}$-enrichment of $\Ho \mathcal{M}_\mathrm{cf}$: though the hom-categories are determined uniquely up to unique isomorphism, it appears to be impossible to unambiguously define the horizontal composition without reference to a known 2-category.
\end{remark}

\section{Segal spaces}

We now turn to Rezk's theory of $\tuple{\infty, 1}$-categories. As usual, $\cat{\Simplex}$ denotes the full subcategory of $\cat{\Cat}$ spanned by the categories $\bracket{0}, \bracket{1}, \bracket{2}, \ldots$, where $\bracket{n}$ is the category freely generated by the following graph:
\[
0 \to 1 \to \cdots \to \parens{n-1} \to n
\]
For consistency, we will follow the conventions of \citep{Joyal-Tierney:2007} regarding bisimplicial sets:

\begin{dfn}
A \strong{bisimplicial set} is a functor $\op{\cat{\Simplex}} \times \op{\cat{\Simplex}} \to \cat{\Set}$. For brevity, if $X$ is a bisimplicial set, we write $X_{n,m}$ for the set $X \argp{\bracket{n}, \bracket{m}}$. The \strong{$n$-th column} of a bisimplicial set $X$ is the simplicial set $X_{n, \bullet}$, and the \strong{$m$-th row} of $X$ is the simplicial set $X_{\bullet, m}$.
\end{dfn}

\makenumpar
Given simplicial sets $X$ and $Y$, we define a bisimplicial set $X \boxtimes Y$ by the following formula:
\[
\parens{X \boxtimes Y}_{n,m} = X_n \times Y_m
\]
It is not hard to check that the bisimplicial set that represents $\tuple{\bracket{n}, \bracket{m}}$ is (isomorphic to) $\Delta^n \boxtimes \Delta^m$. Note also that $X \boxtimes Y$ is naturally isomorphic to $\parens{X \boxtimes \Delta^0} \times \parens{\Delta^0 \boxtimes Y}$.

\begin{prop}
\label{prop:bisimplicial.set.properties}
\needspace{2.5\baselineskip}
Let $\cat{\SSSet}$ be the category of bisimplicial sets.
\begin{enumerate}[(i)]
\item $\cat{\SSSet}$ is a cartesian closed category, with exponential objects defined by the formula below:
\[
\xHom{X}{Y}_{n,m} = \Hom[\SSSet]{\parens{\Delta^n \boxtimes \Delta^m} \times X}{Y}
\]

\item $\cat{\SSSet}$ admits (at least) two $\cat{\SSet}$-enrichments: the \strong{vertical enrichment}, where the space of morphisms $X \to Y$ is given by $\xHom{X}{Y}_{0, \bullet}$, and the \strong{horizontal enrichment}, where the space of morphisms $X \to Y$ is given by $\xHom{X}{Y}_{\bullet, 0}$.
%
%\item The functor $\upK_\uph : \cat{\SSet} \to \cat{\SSSet}$ defined by $X \mapsto X \boxtimes \Delta^0$ is fully faithful; symmetrically, the functor $\upK_\upv : \cat{\SSet} \to \cat{\SSSet}$ defined by $X \mapsto \Delta^0 \boxtimes X$ is fully faithful.
%
%\item The functor $\cat{\SSSet} \to \cat{\SSet}$ defined by $X \mapsto X_{n, \bullet}$ can be represented by $\upK_\uph \Delta^n$ with respect to the vertical enrichment; symmetrically, the functor $\cat{\SSSet} \to \cat{\SSet}$ defined by $X \mapsto X_{\bullet, m}$ can be represented by $\upK_\upv \Delta^m$ with respect to the horizontal enrichment. 

\item For each simplicial set $Y$, the functor $\cat{\SSet} \to \cat{\SSSet}$ defined by $X \mapsto X \boxtimes Y$ has a right adjoint, namely the functor that sends a bisimplicial set $Z$ to the simplicial set $\upM_\uph \argp{Y, Z}$ defined by the formula below:
\[
\parens{\upM_\uph \argp{Y, Z}}_n = \Hom[\SSet]{Y}{Z_{n, \bullet}}
\]
Symmetrically, for each simplicial set $X$, the functor $\cat{\SSet} \to \cat{\SSSet}$ defined by $Y \mapsto X \boxtimes Y$ has a right adjoint, namely the functor that sends a bisimplicial set $Z$ to the simplicial set $\upM_\upv \argp{X, Z}$ defined by the formula below:
\[
\parens{\upM_\upv \argp{X, Z}}_m = \Hom[\SSet]{X}{Z_{\bullet, m}} 
\]

\item For each simplicial set $Y$, the functor $\upM_\uph \argp{Y, \blank} : \cat{\SSSet} \to \cat{\SSet}$ can be represented by $\upK_\upv Y$ with respect to the horizontal enrichment; symmetrically, for each simplicial set $X$, the functor $\upM_\upv \argp{X, \blank} : \cat{\SSSet} \to \cat{\SSet}$ can be represented by $\upK_\uph X$ with respect to the vertical enrichment.

\item The tensor product of a simplicial set $X$ and a bisimplicial set $Z$ with respect to the vertical enrichment of $\cat{\SSSet}$ is given by $\upK_\upv X \times Z$; symmetrically, the tensor product of a simplicial set $X$ and a bisimplicial set $Z$ with respect to the horizontal enrichment of $\cat{\SSet}$ is given by $\upK_\uph X \times Z$.
\end{enumerate}
\end{prop}
\begin{proof} \exerproof
Straightforward, but omitted.
\end{proof}

\begin{thm}[Reedy]
\needspace{3.0\baselineskip}
The following data define a cartesian model structure on $\cat{\SSSet}$:
\begin{itemize}
\item The weak equivalences are \strong{vertical weak homotopy equivalences}, \ie the morphisms that induce weak homotopy equivalences between the respective columns.

\item The cofibrations are the monomorphisms.
\end{itemize}
We refer to this as the \strong{vertical Reedy model structure} on $\cat{\SSSet}$, and the fibrations are called \strong{vertical Reedy fibrations}. It is a simplicial model structure with respect to the vertical enrichment of $\cat{\SSSet}$.
\end{thm}
\begin{proof} \openproof
% See Theorems 15.3.4 and 15.8.7 in \citep{Hirschhorn:2003}, or 
See Theorem 2.6 in \citep{Joyal-Tierney:2007}. 
\end{proof}

\begin{lem}
\label{lem:vertically.Reedy-fibrant.bisimplicial.sets}
Let $X$ be a bisimplicial set. The following are equivalent:
\begin{enumerate}[(i)]
\item $X$ is fibrant with respect to the vertical Reedy model structure.

\item For all $n \ge 0$, the morphism $\upM_\upv \argp{\Delta^n, X} \to \upM_\upv \argp{\partial \Delta^n, X}$ induced by the boundary inclusion $\partial \Delta^n \embedinto \Delta^n$ is a Kan fibration.

\item For all monomorphisms $i : Y \to Z$ in $\cat{\SSet}$, the morphism
\[
\upM_\upv \argp{i, X} : \upM_\upv \argp{Z, X} \to \upM_\upv \argp{Y, X}
\]
induced by $i$ is a Kan fibration.
\end{enumerate}
\end{lem}
\begin{proof} \openproof
See Proposition 2.3 in \citep{Joyal-Tierney:2007}
%, or Lemma 16.4.4 in \citep{Hirschhorn:2003}.
\end{proof}

\begin{dfn}
Let $n \ge 1$. A \strong{principal edge} of $\Delta^n$ is a 1-simplex in $\Delta^n$ that corresponds to a functor $\bracket{1} \to \bracket{n}$ sending $0$ to $i$ and $1$ to $i + 1$ (where $0 \le i < n$). The \strong{spine} of $\Delta^n$ is the smallest simplicial subset $G^n \subseteq \Delta^n$ containing all the principal edges in $\Delta^n$.
\end{dfn} 

\begin{dfn}
\needspace{3.0\baselineskip}
A bisimplicial set $X$ satisfies the \strong{Segal condition} (\resp \strong{strict Segal condition}) if it has the following property:
\begin{itemize}
\item For $n \ge 1$, the morphism $\upM_\upv \argp{\Delta^n, X} \to \upM_\upv \argp{G^n, X}$ induced by the spine inclusion $G^n \embedinto \Delta^n$ is a weak equivalence (\resp isomorphism) of simplicial sets.
\end{itemize}
A \strong{Segal space} is a bisimplicial set that is fibrant with respect to the vertical Reedy model structure and satisfies the Segal condition.
\end{dfn}

\begin{remark}
\label{rem:Segal.operators.are.trivial.Kan.fibrations}
If $X$ is a Segal space, then \autoref{lem:vertically.Reedy-fibrant.bisimplicial.sets} implies the morphism $\upM_\upv \argp{\Delta^n, X} \to \upM_\upv \argp{G^n, X}$ is a trivial Kan fibration; in particular, it is a split epimorphism.
\end{remark}
%
%\begin{lem}
%Let $X$ and $Y$ be bisimplicial sets. If $Y$ satisfies the strict Segal condition, then $Y$ is 2-coskeletal in the sense that morphisms $X \to Y$ are freely and uniquely determined by their restrictions to the 0th, 1st, and 2nd columns.
%\end{lem}
%\begin{proof} \exerproof
%The usual proof that the nerve of a category is a 2-coskeletal simplicial set can be adapted for this situation.
%\end{proof}

\begin{dfn}
The \strong{classifying diagram} of a small category $\mathcal{C}$ is the bisimplicial set $\Nv{\mathcal{C}}$ defined below,
\[
\Nv{\mathcal{C}}_{n,m} = \Hom[\Cat]{\bracket{n} \times \mbfI \bracket{m}}{\mathcal{C}}
\]
where $\mbfI \bracket{m}$ is the groupoid obtained by freely inverting all morphisms in  $\bracket{m}$.
\end{dfn}

\begin{thm}[Rezk]
\label{thm:classifying.diagram.functor.homotopical.properties}
\needspace{3.0\baselineskip}
Let $\Nv : \cat{\Cat} \to \cat{\SSSet}$ be the functor that sends a small category to its classifying diagram.
\begin{enumerate}[(i)]
\item The functor $\Nv : \cat{\Cat} \to \cat{\SSSet}$ is fully faithful, cartesian closed, and for any functor $F : \mathcal{C} \to \mathcal{D}$, $\Nv{F} : \Nv{\mathcal{C}} \to \Nv{\mathcal{D}}$ is a vertical weak homotopy equivalence if and only if $F : \mathcal{C} \to \mathcal{D}$ is a categorical equivalence.

\item For any small category $\mathcal{C}$, the classifying diagram $\Nv{\mathcal{C}}$ is a Segal space.

\item $\Nv$ has a left adjoint, $\tau_1 : \cat{\SSSet} \to \cat{\Cat}$, which is the unique (up to unique isomorphism) colimit-preserving functor that sends $\Delta^n \boxtimes \Delta^m$ to the category $\bracket{n} \times \mbfI \bracket{m}$. 
\end{enumerate}
\end{thm}
\begin{proof} \openproof
(i). See Theorem 3.6 in \citep{Rezk:2001}.

\bigskip\noindent
(ii). It is not hard to see that $\Nv{\mathcal{C}}$ satisfies the (strict!) Segal condition. For the fibrancy of $\Nv{\mathcal{C}}$ with respect to the vertical Reedy model structure, see Lemma 3.8 in \citep{Rezk:2001}. 

\bigskip\noindent
(iii). As usual, one may either apply the accessible adjoint functor theorem or use the theory of Kan extensions. To see that $\tau_1  \argp{\Delta^n \boxtimes \Delta^m}$ is (isomorphic to) $\bracket{n} \times \mbfI \bracket{m}$, simply observe that the Yoneda lemma gives us the following natural bijection:
\[
\Hom[\SSSet]{\Delta^n \boxtimes \Delta^m}{\nv{\mathcal{C}}} \cong \nv{\mathcal{C}}_{n,m} = \Hom[\Cat]{\bracket{n} \times \mbfI \bracket{m}}{\mathcal{C}}
\qedhere
\]
\end{proof}

\begin{prop}
\label{prop:classifying.diagram.functor.properties}
Let $\mathcal{C}$ be a small category.
\begin{enumerate}[(i)]
\item There is a natural isomorphism $\nv{\iso \xHom{\bracket{n}}{\mathcal{C}}} \cong \Nv{\mathcal{C}}_{n, \bullet}$, where $\iso : \cat{\Cat} \to \cat{\Grpd}$ is the \emph{right} adjoint of the inclusion $\cat{\Grpd} \embedinto \cat{\Cat}$.

\item There is a natural isomorphism $\Nv{\xHom{\bracket{n} \times \mbfI \bracket{m}}{\mathcal{C}}} \cong \xHom{\Delta^n \boxtimes \Delta^m}{\Nv{\mathcal{C}}}$.

\item The functor $\tau_1 : \cat{\SSSet} \to \cat{\Cat}$ preserves finite products, and for any bisimplicial set $X$, the natural morphism
\[
\Nv{\xHom{\tau_1 X}{\mathcal{C}}} \to \xHom{\Nv{\tau_1 X}}{\Nv{\mathcal{C}}} \to \xHom{X}{\Nv{\mathcal{C}}}
\]
is an isomorphism of bisimplicial sets.
\end{enumerate}
\end{prop}
\begin{proof}
(i). We have the following natural bijections:
\begin{align*}
\nv{\iso \xHom{\bracket{n}}{\mathcal{C}}}_m
& = \Hom[\Cat]{\bracket{m}}{\iso \xHom{\bracket{n}}{\mathcal{C}}} \\
& \cong \Hom[\Cat]{\mbfI \bracket{m}}{\xHom{\bracket{n}}{\mathcal{C}}} \\
& \cong \Hom[\Cat]{\bracket{n} \times \mbfI \bracket{m}}{\mathcal{C}} \\
& = \Nv{\mathcal{C}}_{n,m}
\end{align*}
Thus, we have the desired natural isomorphism of simplicial sets.

\bigskip\noindent
(ii). To prove the claim, it suffices to show that we have a natural bijection of the form below:
\[
\Hom[\SSSet]{\parens{\Delta^k \boxtimes \Delta^l} \times \parens{\Delta^n \boxtimes \Delta^m}}{\Nv{\mathcal{C}}} \cong \Hom[\Cat]{\bracket{k} \times \mbfI \bracket{l} \times \bracket{n} \times \mbfI \bracket{m}}{\mathcal{C}}
\]
We now use the calculus of ends to establish such a natural bijection:
\begin{align*}
\mathrlap{\Hom[\SSSet]{\parens{\Delta^k \boxtimes \Delta^l} \times \parens{\Delta^n \boxtimes \Delta^m}}{\Nv{\mathcal{C}}}}\hspace{6.0em} \\
& \cong \int_{\bracket{p} : \cat{\Simplex}} \int_{\bracket{q} : \cat{\Simplex}} \Hom[\Set]{\Delta^k_p \times \Delta^l_q \times \Delta^n_p \times \Delta^m_q}{\Nv{\mathcal{C}}_{p,q}} \\
& \cong \int_{\bracket{p} : \cat{\Simplex}} \Hom[\Set]{\Delta^k_p \times \Delta^n_p}{\int_{\bracket{q} : \cat{\Simplex}} \Hom[\Set]{\Delta^l_q \times \Delta^m_q}{\Nv{\mathcal{C}}_{p,q}}} \\
& \cong \int_{\bracket{p} : \cat{\Simplex}} \Hom[\Set]{\Delta^k_p \times \Delta^n_p}{\Hom[\SSet]{\Delta^l \times \Delta^m}{\Nv{\mathcal{C}}_{p,\bullet}}} \\
& \cong \int_{\bracket{p} : \cat{\Simplex}} \Hom[\Set]{\Delta^k_p \times \Delta^n_p}{\Hom[\SSet]{\Delta^l \times \Delta^m}{\nv{\iso \Func{\bracket{p}}{\mathcal{C}}}}} \\
& \cong \int_{\bracket{p} : \cat{\Simplex}} \Hom[\Set]{\Delta^k_p \times \Delta^n_p}{\Hom[\Cat]{\bracket{l} \times \bracket{m}}{\iso \Func{\bracket{p}}{\mathcal{C}}}} \\
& \cong \int_{\bracket{p} : \cat{\Simplex}} \Hom[\Set]{\Delta^k_p \times \Delta^n_p}{\Hom[\Cat]{\mbfI \bracket{l} \times \mbfI \bracket{m}}{\Func{\bracket{p}}{\mathcal{C}}}} \\
& \cong \int_{\bracket{p} : \cat{\Simplex}} \Hom[\Set]{\Delta^k_p \times \Delta^n_p}{\Hom[\Cat]{\bracket{p}}{\Func{\mbfI \bracket{l} \times \mbfI \bracket{m}}{\mathcal{C}}}} \\
& \cong \Hom[\SSet]{\Delta^k \times \Delta^n}{\nv{\Func{\mbfI \bracket{l} \times \mbfI \bracket{m}}{\mathcal{C}}}} \\
& \cong \Hom[\Cat]{\bracket{k} \times \bracket{n}}{\Func{\mbfI \bracket{l} \times \mbfI \bracket{m}}{\mathcal{C}}} \\
& \cong \Hom[\Cat]{\bracket{k} \times \mbfI \bracket{l} \times \bracket{n} \times \mbfI \bracket{m}}{\mathcal{C}}
\end{align*}

\bigskip\noindent
(iii). We first show that $\xHom{X}{\Nv{\mathcal{C}}}$ is isomorphic to the classifying diagram of \emph{some} category. Indeed, every bisimplicial set is canonically a colimit of a canonical small diagram of representable bisimplicial sets, and the functor $\xHom{\blank}{\Nv{\mathcal{C}}} : \op{\cat{\SSSet}} \to \cat{\SSSet}$ is a right adjoint, so we deduce that $\xHom{X}{\Nv{\mathcal{C}}}$ is the limit of a small diagram of bisimplicial sets of the form $\xHom{\Delta^n \boxtimes \Delta^m}{\Nv{\mathcal{C}}}$, which by claim (ii) is naturally isomorphic to $\Nv{\Func{\bracket{n} \times \mbfI \bracket{m}}{\mathcal{C}}}$. Since $\cat{\Cat}$ is a complete category and $\Nv : \cat{\Cat} \to \cat{\SSSet}$ is a right adjoint, it follows that $\xHom{X}{\Nv{\mathcal{C}}}$ is isomorphic to the classifying diagram of a small category.

Thus, $\Nv$ embeds $\cat{\Cat}$ as a reflective exponential ideal in $\cat{\SSSet}$, and we may deduce the main claims by applying \autoref{prop:reflective.exponential.ideals}.
\end{proof}

\makenumpar
The category $\tau_1 X$ admits a simple explicit description when $X$ is a Segal space. First, for each tuple $\tuple{x_0, \ldots, x_n}$ of vertices in $X_{0, \bullet}$, let $X \argp{x_0, \ldots, x_n}$ be the simplicial set defined by the pullback diagram in $\cat{\SSet}$ shown below,
\[
\begin{tikzcd}
X \argp{x_0, \ldots, x_n} \dar \rar &
X_{n, \bullet} \dar \\
\Delta^0 \rar[swap]{\prodtuple{x_0, \ldots, x_n}} &
X_{0, \bullet} \times \cdots \times X_{0, \bullet}
\end{tikzcd}
\]
where the morphism $X_{n, \bullet} \to X_{0, \bullet} \times \cdots \times X_{0, \bullet}$ is the one induced by the obvious inclusion of vertices $\Delta^0 \amalg \cdots \amalg \Delta^0 \to \Delta^n$. Note that there is a unique morphism $X \argp{x_0, \ldots, x_n} \to X \argp{x_0, x_1} \times \cdots \times X \argp{x_{n-1}, x_n}$ making the diagram below a pullback square in $\cat{\SSet}$,
\[
\begin{tikzcd}
X \argp{x_0, \ldots, x_n} \dar[hookrightarrow] \rar &
X \argp{x_0, x_1} \times \cdots \times X \argp{x_{n-1}, x_n} \dar[hookrightarrow] \\
\upM_\upv \argp{\Delta^n, X} \rar &
\upM_\upv \argp{G^n, X}
\end{tikzcd}
\]
and since the bottom arrow in the diagram is a trivial Kan fibration, so too is the top arrow. We should think of $X \argp{x_0, x_1}$ as the space of morphisms $x_0 \to x_1$; the term `space' is reasonable because $X \argp{x_0, x_1}$ is a Kan complex. Consider the following category $\Ho X$:
\begin{itemize}
\item The objects are the vertices of the 0th column $X_{0, \bullet}$.

\item The morphisms $x_0 \to x_1$ are the connected components of $X \argp{x_0, x_1}$.

\item The identity morphism $x \to x$ is the connected component of the vertex $s_{0, \bullet} \argp{x}$ in $X \argp{x, x}$, where $s_{0, \bullet}$ is the degeneracy operator $X_{0, \bullet} \to X_{1, \bullet}$.

\item Composition is induced by the face operator $d_{1, \bullet} : X_{2, \bullet} \to X_{1, \bullet}$. More precisely, $d_{1, \bullet}$ induces a morphism $X \argp{x_0, x_1, x_2} \to X \argp{x_0, x_2}$, and we have a weak homotopy equivalence $X \argp{x_0, x_1, x_2} \to X \argp{x_0, x_1} \times X \argp{x_1, x_2}$, so by applying the connected components functor $\pi_0 : \cat{\SSet} \to \cat{\Set}$, we obtain a well-defined map $\Hom[\Ho X]{x_1}{x_2} \times \Hom[\Ho X]{x_0}{x_1} \to \Hom[\Ho X]{x_0}{x_2}$.
\end{itemize} 

\begin{lem}
The above construction is indeed a category.
\end{lem}
\begin{proof} \openproof
See Proposition 5.4 in \citep{Rezk:2001}.
\end{proof}

\begin{lem}
\label{lem:classifying.diagram.properties}
Let $\mathcal{C}$ be a small category and let $X = \Nv{\mathcal{C}}$.
\begin{enumerate}[(i)]
\item The columns of $X$ are Kan complexes.

\item The vertices of the 0th column $X_{0, \bullet}$ can be canonically identified with objects in $\mathcal{C}$.

\item The vertices of $X \argp{x_0, x_1}$ can be canonically identified with morphisms $x_0 \to x_1$ in $\mathcal{C}$, and $X \argp{x_0, x_1}$ is moreover discrete as a simplicial set.
\end{enumerate}
\end{lem}
\begin{proof}
(i). It was shown in \autoref{prop:classifying.diagram.functor.properties} that the columns of $\Nv{\mathcal{C}}$ are isomorphic to nerves of groupoids, and it is well-known that the nerve of any groupoid is a Kan complex.

\bigskip\noindent
(ii). By definition, $\Nv{\mathcal{C}}_{0,0} = \Hom[\Cat]{\bracket{0} \times \mbfI \bracket{0}}{\mathcal{C}}$, and there is a natural bijection between the latter and the set of objects in $\mathcal{C}$.

\bigskip\noindent
(iii). The $n$-simplices of $\Hom[X]{x_0}{x_1}$ correspond to commutative diagrams of the following form in $\mathcal{C}$,
\[
\begin{tikzcd}
x_0 \dar \rar{\id} &
x_0 \dar \rar{\id} &
\cdots \rar{\id} &
x_0 \dar \rar{\id} &
x_0 \dar \\
x_1 \rar[swap]{\id} &
x_1 \rar[swap]{\id} &
\cdots \rar[swap]{\id} &
x_1 \rar[swap]{\id} &
x_1
\end{tikzcd}
\]
where the number of columns of horizontal arrows is $n$. Thus, there is a canonical identification of vertices of $\Hom[X]{x_0}{x_1}$ with morphisms $x_0 \to x_1$ in $\mathcal{C}$, and it is clear that all the $n$-simplices of $\Hom[X]{x_0}{x_1}$ are degenerate for $n \ge 1$.
\end{proof}

\begin{lem}
\label{lem:connections.in.Segal.spaces}
Let $X$ be a Segal space. For each edge $p$ of $X_{0, \bullet}$, there exists an edge $\sigma$ of $X_{1, \bullet}$ satisfying the following equations:
\begin{align*}
d_{1, \bullet} \argp{\sigma} & = p &
d_{0, \bullet} \argp{\sigma} & = s_{\bullet, 0} \argp{d_{\bullet, 0} \argp{p}} &
d_{\bullet, 0} \argp{\sigma} & = s_{0, \bullet} \argp{d_{\bullet, 0} \argp{p}}
\end{align*}
Moreover, the connected component of $X_{1, \bullet}$ containing the vertex $d_{\bullet, 1} \argp{\sigma}$ depends only on $p$, and this defines a functor $j_X : \tau_1 \argp{X_{0, \bullet}} \to \Ho X$.
\end{lem}
\begin{proof}
Let $U = \parens{\partial \Delta^1 \boxtimes \Delta^1} \cup \parens{\Delta^1 \boxtimes \Lambda^1_1} \subseteq \Delta^1 \boxtimes \Delta^1$. It is not hard to see that the inclusion $U \embedinto \Delta^1 \boxtimes \Delta^1$ is a trivial cofibration with respect to the vertical Reedy model structure: the columns of $U$ and $\Delta^1 \boxtimes \Delta^1$ are disjoint unions of contractible simplicial sets, and the inclusion $U \embedinto \Delta^1 \boxtimes \Delta^1$ induces a bijection of connected components in each column. The existence of $\sigma$ amounts to saying that morphisms $U \to X$ of a particular form can be extended along the inclusion $U \embedinto \Delta^1 \boxtimes \Delta^1$ in a homotopically unique way, and this is a straightforward application of axiom SM7 for simplicial model categories. Similarly, the claim that this construction defines a functor $\tau_1 X_{0, \bullet} \to \Ho X$ follows from the fact that the inclusion
\[
\parens{\set{0} \boxtimes \Delta^2} \cup \parens{\set{2} \boxtimes \Delta^2} \cup \parens{\Delta^2 \boxtimes \set{2}} \cup \parens{\Lambda^2_1 \boxtimes \Lambda^2_1} \hookrightarrow \Delta^2 \boxtimes \Delta^2 
\]
is a trivial cofibration with respect to the vertical Reedy model structure.
\end{proof}

\begin{remark}
In the case where $X = \Nv{\mathcal{C}}$ for a small category $\mathcal{C}$, the lemma says that, for every isomorphism $p : x_0 \to x_1$ in $\mathcal{C}$, there exists a morphism $f : x_0 \to x_1$ in $\mathcal{C}$ making the diagram below commute,
\[
\begin{tikzcd}
x_0 \dar[swap]{f} \rar{p} &
x_1 \dar{\id} \\
x_1 \rar[swap]{\id} &
x_1
\end{tikzcd}
\]
and this $f$ is moreover unique up to isomorphism. Indeed, in this situation, we must have $f = p$. The functoriality of the map $p \mapsto f$ reduces to the fact that we can paste commutative squares together as in the diagram below:
\[
\begin{tikzcd}
x_0 \dar[swap]{f_1} \rar{p_1} &
x_1 \dar[swap]{\id} \rar{p_2} &
x_2 \dar{\id} \\
x_1 \dar[swap]{f_2} \rar{\id} &
x_1 \dar[swap]{f_2} \rar{p_2} &
x_2 \dar{\id} \\
x_2 \rar[swap]{\id} &
x_2 \rar[swap]{\id} &
x_2
\end{tikzcd}
\]
\end{remark}

\begin{prop}
\label{prop:truncation.as.homotopy.category}
Let $X$ be a Segal space. For each small category $\mathcal{C}$, there is a bijection
\[
\Hom[\SSSet]{X}{\Nv{\mathcal{C}}} \cong \Hom[\Cat]{\Ho X}{\mathcal{C}}
\]
and it is natural in $\mathcal{C}$. In particular, $\Ho X$ is isomorphic to $\tau_1 X$.
\end{prop}
\begin{proof}
Suppose we are given $F : X \to \Nv{\mathcal{C}}$. We define the functor $\bar{F} : \Ho X \to \mathcal{C}$ on objects by the map $F_{0,0} : X_{0,0} \to \Nv{\mathcal{C}}_{0,0}$, and \autoref{lem:classifying.diagram.properties} implies $F_{1, \bullet} : X_{1, \bullet} \to \Nv{\mathcal{C}}_{1, \bullet}$ induces well-defined maps $\Hom[\Ho X]{x_0}{x_1} \to \Hom[\mathcal{C}]{\bar{F} x_0}{\bar{F} x_1}$ that preserve the identity and composition, as required for a functor. The map $F \mapsto \bar{F}$ is clearly natural in $\mathcal{C}$, so to prove the proposition it is enough to construct an inverse for $F \mapsto \bar{F}$.

Let $\mathcal{A}_{\bullet}$ and $\mathcal{B}_{\bullet}$ be the simplicial objects in $\cat{\Cat}$ defined below:
\begin{align*}
\mathcal{A}_n & = \tau_1 \argp{X_{n, \bullet}} &
\mathcal{B}_n & = \iso \Func{\bracket{n}}{\Ho X}
\end{align*}
Observe that \autoref{lem:vertically.Reedy-fibrant.bisimplicial.sets} implies that the columns $X_{n, \bullet}$ are all Kan complexes, so the categories $\mathcal{A}_n$ are all groupoids. Thus $\mathcal{A}_{\bullet}$ and $\mathcal{B}_{\bullet}$ are in fact simplicial objects in $\cat{\Grpd}$. We previously showed in \autoref{prop:classifying.diagram.functor.properties} that the degreewise nerve of $\mathcal{B}_{\bullet}$ is isomorphic to $\Nv{\Ho X}$, so the adjunction $\tau_1 \dashv \nv$ yields a bijection between morphisms $X \to \Nv{\Ho X}$ and simplicial functors $\mathcal{A}_{\bullet} \to \mathcal{B}_{\bullet}$. It is not hard to see that $\mathcal{B}_{\bullet}$ satisfies the strict Segal condition and is therefore 2-coskeletal as a simplicial object in $\cat{\Grpd}$; thus simplicial functors $\mathcal{A}_{\bullet} \to \mathcal{B}_{\bullet}$ are freely and uniquely determined by their restrictions to degrees $\le 2$.

We define $P_0 : \mathcal{A}_0 \to \mathcal{B}_0$ by factoring the functor $j_X : \tau_1 \argp{X_{0, \bullet}} \to \Ho X$ of \autoref{lem:connections.in.Segal.spaces} using the universal property of $\mathcal{B}_0$ as the maximal subgroupoid of $\mathcal{C}$. There is then a unique functor $P_1 : \mathcal{A}_1 \to \mathcal{B}_1$ satisfying these conditions:
\begin{itemize}
\item For each vertex $f$ of $X \argp{x_0, x_1}$, $P_1$ sends the object $f$ in $\mathcal{A}_1$ to the object $\bracket{f}$ in $\mathcal{B}_1$, where $\bracket{f}$ denotes the connected component of $f$ in $X \argp{x_0, x_1}$.

\item We have $d_1 P_1 = P_0 d_1$ and $d_0 P_1 = P_0 d_0$ as functors.
\end{itemize}
This is because for any edge $\sigma$ of $X_{1, \bullet}$, we have
\[
\bracket{d_{\bullet, 0} \argp{\sigma}} \circ j_X \argp{\bracket{d_{1, \bullet} \argp{\sigma}}} = j_X \argp{\bracket{d_{0, \bullet} \argp{\sigma}}} \circ \bracket{d_{\bullet, 1} \argp{\sigma}}
\]
as morphisms in $\Ho X$. Since $\mathcal{B}_{\bullet}$ satisfies the strict Segal condition, there is at most one functor $P_2 : \mathcal{A}_2 \to \mathcal{B}_2$ compatible with $P_0$, $P_1$, and the 2-truncated simplicial structures of $\mathcal{A}_{\bullet}$ and $\mathcal{B}_{\bullet}$; such a $P_2$ \emph{does} exist by definition of composition in $\Ho X$. We therefore obtain a simplicial functor $P_{\bullet} : \mathcal{A}_{\bullet} \to \mathcal{B}_{\bullet}$, and hence a morphism $Q : X \to \Nv{\Ho X}$ by transposing degreewise across the adjunction $\tau_1 \dashv \nv$.

It is clear that the functor $\bar{Q} : \Ho X \to \Ho X$ induced by the morphism $Q : X \to \Nv{\Ho X}$ is $\id : \Ho X \to \Ho X$. To complete the proof of the proposition, we must show that we have $F = \Nv{\bar{F}} \circ Q$ for all small categories $\mathcal{C}$ and all morphisms $F : X \to \Nv{\mathcal{C}}$. Since the columns of $\Nv{\mathcal{C}}$ are nerves of groupoids, we can transpose across the adjunction $\tau_1 \dashv \nv$ and check the equation in the category of simplicial objects in $\cat{\Grpd}$; this is a straightforward matter of applying the methods employed in the previous paragraph to construct $P_{\bullet} : \mathcal{A}_{\bullet} \to \mathcal{B}_{\bullet}$. 
\end{proof}

\section{Complete Segal spaces}

Although the theory of Segal spaces looks very much like a homotopy-coherent theory of categories internal to Kan complexes, it is not as well-behaved as one would hope: the obvious notion of equivalence of Segal spaces (namely, Dwyer--Kan equivalences) does not agree with the weak equivalences obtained by taking the left Bousfield localisation of the vertical Reedy model structure with respect to Segal spaces. To remedy this, Rezk introduced a notion of `completeness' that connects the categorical structure of a Segal space $X$ with the homotopical structure of the Kan complex $X_{0, \bullet}$.

\begin{dfn}
A \strong{Dwyer--Kan equivalence} is a morphism $F : X \to Y$ between Segal spaces with the following properties:
\begin{itemize}
\item The induced functor $\tau_1 F : \tau_1 X \to \tau_1 Y$ is a categorical equivalence.

\item For each pair $\tuple{x_0, x_1}$ of vertices of $X_{0, \bullet}$, the hom-space morphisms $X \argp{x_0, x_1} \to Y \argp{F x_0, F x_1}$ induced by $F : X \to Y$ are weak homotopy equivalences in $\cat{\SSet}$.
\end{itemize}
\end{dfn}

In \autoref{lem:connections.in.Segal.spaces}, we saw that paths in the 0th column of a Segal space $X$ can be mapped functorially to isomorphisms in $\tau_1 X$. Informally speaking, a complete Segal space is a Segal space where this correspondence is reversible; the purpose of the following paragraphs is to make this precise.

\begin{dfn}
An \strong{equivalence} in a Segal space $X$ is a vertex $f$ of $X_{1, \bullet}$ such that $\bracket{f}$ is an isomorphism in $\tau_1 X$. We write $X_\mathrm{eq}$ for the maximal simplicial subset of $X_{1, \bullet}$ whose vertices are precisely the equivalences in $X$.
\end{dfn}

\begin{prop}
Let $X$ be a Segal space.
\begin{enumerate}[(i)]
\item If $f_0$ and $f_1$ are two vertices in the same connected component of $X_{1, \bullet}$, then $f_0$ is an equivalence if and only if $f_1$ is an equivalence.

\item $X_\mathrm{eq}$ is a union of connected components of $X_{1, \bullet}$.

\item The degeneracy operator $s_{0, \bullet} : X_{0, \bullet} \to X_{1, \bullet}$ factors through the inclusion $X_\mathrm{eq} \embedinto X_{1, \bullet}$.
\end{enumerate}
\end{prop}
\begin{proof} \openproof
See Lemma 5.8 in \citep{Rezk:2001}.
\end{proof}

\begin{dfn}
A \strong{complete Segal space} is a Segal space $X$ such that the morphism $X_{0, \bullet} \to X_\mathrm{eq}$ is a weak homotopy equivalence in $\cat{\SSet}$. A \strong{weak CSS equivalence} is a morphism $F : X \to Y$ in $\cat{\SSSet}$ such that the induced morphisms
\[
\xHom{F}{Z}_{0, \bullet} : \xHom{Y}{Z}_{0, \bullet} \to \xHom{X}{Z}_{0, \bullet}
\]
are weak homotopy equivalences in $\cat{\SSet}$ for all complete Segal spaces $Z$.
\end{dfn}

\begin{lem}
\label{lem:completeness.and.univalence}
If $X$ is a complete Segal space, then there exists a commutative diagram in $\cat{\SSet}$ of the form below,
\[
\begin{tikzcd}
X_{0, \bullet} \dar \rar &
\xHom{\Delta^1}{X_{0, \bullet}} \dar \\
X_\mathrm{eq} \urar[dashed] \rar &
X_{0, \bullet} \times X_{0, \bullet}
\end{tikzcd}
\]
where the indicated arrow $X_\mathrm{eq} \to \xHom{\Delta^1}{X_{0, \bullet}}$ is a weak homotopy equivalence.
\end{lem}
\begin{proof}
For any Segal space $X$, the morphism $X_{0, \bullet} \to X_\mathrm{eq}$ is a split monomorphism, so $X$ is a complete Segal space if and only if $X_{0, \bullet} \to X_\mathrm{eq}$ is a trivial cofibration with respect to the Kan--Quillen model structure. Since the morphism $\xHom{\Delta^1}{X_{0, \bullet}} \to \xHom{\partial \Delta^1}{X_{0, \bullet}} \cong X_{0, \bullet} \times X_{0, \bullet}$ induced by the boundary inclusion $\partial \Delta^1 \embedinto \Delta^1$ is a Kan fibration, the lifting property yields the desired morphism $X_\mathrm{eq} \to \xHom{\Delta^1}{X_{0, \bullet}}$, and it is a weak homotopy equivalence by the 2-out-of-3 property.
\end{proof}

\begin{thm}[Rezk]
The following data define a cartesian model structure on $\cat{\SSSet}$:
\begin{itemize}
\item The weak equivalences are the weak CSS equivalences.

\item The cofibrations are the monomorphisms.
\end{itemize}
We refer to this as the \strong{Rezk model structure for complete Segal spaces}, and the fibrant objects in this model structure are the complete Segal spaces.
\end{thm}
\begin{proof} \openproof
See Theorem 7.2 in \citep{Rezk:2001}.
\end{proof}

\begin{remark}
It follows from the definition of weak CSS equivalence and axiom SM7 for simplicial model structures that every vertical weak homotopy equivalence in $\cat{\SSSet}$ is also a weak CSS equivalence. Thus the Rezk model structure for complete Segal spaces is a left Bousfield localisation of the vertical Reedy model structure.
\end{remark}

\begin{lem}
\label{lem:equivalences.and.homotopies}
Let $X$ be a complete Segal space, and let $x_0$ and $x_1$ be vertices of $X_{0, \bullet}$. The following are equivalent:
\begin{enumerate}[(i)]
\item $x_0$ and $x_1$ are isomorphic as objects in $\tau_1 X$.

\item $x_0$ and $x_1$ are in the same connected component of $X_{0, \bullet}$.

\item $x_0$ and $x_1$ are left homotopic as morphisms $\Delta^0 \boxtimes \Delta^0 \to X$ in the Rezk model structure for complete Segal spaces.
\end{enumerate}
\end{lem}
\begin{proof}
(i) \iff (ii). This is a corollary of \autoref{lem:completeness.and.univalence}.

\bigskip\noindent
(ii) \implies (iii). The unique morphism $\Delta^0 \boxtimes \Delta^1 \to \Delta^0 \boxtimes \Delta^0$ is a vertical weak homotopy equivalence, so it is a weak CSS equivalence \emph{a fortiori}. This shows that $\Delta^0 \boxtimes \Delta^1$ (with the evident inclusions and projection) is a cylinder object for $\Delta^0 \boxtimes \Delta^0$ in the Rezk model structure for complete Segal spaces. Since $X_{0, \bullet}$ is a Kan complex, $x_0$ and $x_1$ are in the same connected component if and only if there exists an edge $p : x_0 \to x_1$ in $X_{0, \bullet}$. Thus, we obtain a morphism $\Delta^0 \boxtimes \Delta^1 \to X$ in $\cat{\SSSet}$, \ie a left homotopy from $x_0$ to $x_1$ considered as morphisms $\Delta^0 \boxtimes \Delta^0 \to X$.

\bigskip\noindent
(iii) \implies (ii). We have already seen that $\Delta^0 \boxtimes \Delta^1$ is a cylinder object for $\Delta^0 \boxtimes \Delta^0$; but $X$ is fibrant in the Rezk model structure for complete Segal spaces, so two morphisms $\Delta^0 \boxtimes \Delta^0 \to X$ are left homotopic if and only if they are left homotopic with respect to the cylinder object $\Delta^0 \boxtimes \Delta^1$. (This is a general fact about model categories: see \eg Corollary 1.2.6 in \citep{Hovey:1999}.)
\end{proof}

\begin{thm}[Joyal--Tierney]
\label{thm:Joyal-Tierney.totalisation}
Let $\upJ^m = \nv{\mbfI \bracket{m}}$ and let $t^! : \cat{\SSSet} \to \cat{\SSet}$ be the functor defined by the formula below:
\[
\parens{t^! X}_{n,m} = \Hom[\SSet]{\Delta^n \times \upJ^m}{X}
\]
\begin{enumerate}[(i)]
\item $t^!$ has a left adjoint, $t_! : \cat{\SSSet} \to \cat{\SSet}$, which is the unique (up to unique isomorphism) colimit-preserving functor that sends $\Delta^n \boxtimes \Delta^m$ to the simplicial set $\Delta^n \times \upJ^m$.

\item $\Nv : \cat{\Cat} \to \cat{\SSSet}$ is isomorphic to the composite $t^! \nv$.

\item The adjunction
\[
t_! \dashv t^! : \cat{\SSet} \to \cat{\SSSet}
\]
is a Quillen equivalence with respect to the Joyal model structure for quasicategories and the Rezk model structure for complete Segal spaces.
\end{enumerate}
\end{thm}
\begin{proof} \openproof
(i). As with \autoref{thm:classifying.diagram.functor.homotopical.properties}, use either the accessible adjoint functor theorem or the theory of Kan extensions to construct the left adjoint, and then apply the Yoneda lemma to deduce its action on representable bisimplicial sets.

\bigskip\noindent
(ii). This is an immediate consequence of the fact that $\Delta^n \times \upJ^m$ is naturally isomorphic to $\nv{\bracket{n} \times \mbfI \bracket{m}}$.

\bigskip\noindent
(iii). See Theorem 4.12 in \citep{Joyal-Tierney:2007}.
\end{proof}

\begin{cor}
The adjunction
\[
\tau_1 \dashv {\Nv} : \cat{\Cat} \to \cat{\SSSet}
\]
is a Quillen adjunction with respect to the natural model structure on $\cat{\Cat}$ and the Rezk model structure for complete Segal spaces on $\cat{\SSSet}$.
\end{cor}
\begin{proof}
Quillen adjunctions can be composed, so combine \autoref{prop:truncation-nerve.adjunction.for.quasicategories} and \autoref{thm:Joyal-Tierney.totalisation}.
\end{proof}

\begin{cor}
\label{cor:truncation.and.totalisation}
For all quasicategories $X$, the functor $\tau_1 t_! t^! X \to \tau_1 X$ induced by the adjunction counit $t_! t^! \hoto \id_{\cat{\SSet}}$ is an isomorphism of categories.
\end{cor}
\begin{proof}
The adjunction $t_! \dashv t^!$ is a Quillen equivalence, $X$ is fibrant with respect to the Joyal model structure for quasicategories, and all bisimplicial sets are cofibrant with respect to the Rezk model structure for complete Segal spaces, so the counit component $t_! t^! X \to X$ must be a weak categorical equivalence. Ken Brown's lemma then implies that the induced functor $\tau_1 t_! t^! X \to \tau_1 X$ is a categorical equivalence. However, it is not hard to check that $t_! t^! X \to X$ is always a bijection on vertices, so $\tau_1 t_! t^! X \to \tau_1 X$ must in fact be an isomorphism of categories.
\end{proof}

\begin{prop}
Let $\cat{\CSS}$ be the full subcategory of $\cat{\SSSet}$ spanned by the complete Segal spaces.
\begin{enumerate}[(i)]
\item $\cat{\CSS}$ is closed under small products and exponential objects in $\cat{\SSSet}$; in particular, $\cat{\CSS}$ is a cartesian closed category.

\item Let $\Ho \cat{\CSS}$ be the localisation of $\cat{\CSS}$ at the weak categorical equivalences. The localisation functor $\cat{\CSS} \to \Ho \cat{\CSS}$ preserves small products and exponential objects.

\item If $X$ and $Y$ are complete Segal spaces, then the categorical equivalence class of $\tau_1 \xHom{X}{Y}$ depends only on the weak CSS equivalence classes of $X$ and $Y$.
\end{enumerate}
\end{prop}
\begin{proof}
Apply \autoref{prop:Hurewicz.homotopy.category}.
\end{proof}

\begin{thm}[Rezk]
\label{thm:DK.equivalences.and.weak.CSS.equivalences}
Let $F : X \to Y$ be a morphism between Segal spaces. $F$ is a Dwyer--Kan equivalence between Segal spaces if and only if $F$ is a weak CSS equivalence.
\end{thm}
\begin{proof} \openproof
See Theorem 7.7 in \citep{Rezk:2001}.
\end{proof}

\begin{cor}
\label{cor:equivalences.in.CSS}
Let $F : X \to Y$ be a morphism between complete Segal spaces. The following are equivalent:
\begin{enumerate}[(i)]
\item $F$ is a Dwyer--Kan equivalence.

\item $F$ is a weak CSS equivalence.

\item $F$ is an equivalence in the 2-category $\tau_1 \argb{\ul{\cat{\CSS}}}$.
\end{enumerate}
\end{cor}
\begin{proof}
(i) \iff (ii). This is a special case of \autoref{thm:DK.equivalences.and.weak.CSS.equivalences}.

\bigskip\noindent
(ii) \implies (iii). If $F : X \to Y$ is a weak CSS equivalence, then the morphism $\xHom{F}{Z} : \xHom{Y}{Z} \to \xHom{X}{Z}$ is a weak CSS equivalence for all complete Segal spaces $Z$. Ken Brown's lemma then implies $\tau_1 \xHom{F}{Z} : \tau_1 \xHom{Y}{Z} \to \tau_1 \xHom{X}{Z}$ is a categorical equivalence for all complete Segal spaces $Z$. Hence, $F$ is an equivalence in $\tau_1 \argb{\ul{\cat{\CSS}}}$.

\bigskip\noindent
(iii) \implies (ii). Let $\tau_0 : \cat{\Cat} \to \cat{\Set}$ be the functor that sends a small category $\mathcal{C}$ to the set of isomorphism classes of objects in $\mathcal{C}$. \Autoref{prop:Hurewicz.homotopy.category} and  \autoref{lem:equivalences.and.homotopies} together imply that $\Ho \cat{\CSS} \cong \tau_0 \argb{\tau_1 \argb{\ul{\cat{\CSS}}}}$, so a morphism $F : X \to Y$ in $\cat{\CSS}$ becomes an equivalence in $\tau_1 \argb{\ul{\cat{\CSS}}}$ if and only if $F$ becomes an isomorphism in $\Ho \cat{\CSS}$, \ie if and only if $F$ is a weak CSS equivalence.
\end{proof}

Finally, we are able to prove the claim announced in the introduction:

\begin{thm}
\label{thm:bicategorical.comparison.for.CSS}
\needspace{2.5\baselineskip}
Let $\bicat{\CSS}$ be the 2-category defined by $\tau_1 \argb{\ul{\cat{\CSS}}}$.
\begin{enumerate}[(i)]
\item The functor $t^! : \cat{\SSet} \to \cat{\SSSet}$ restricts to a functor $t^! : \cat{\Qcat} \to \cat{\CSS}$ that preserves small products and weak equivalences.

\item For all quasicategories $X$ and $Y$, the canonical morphism
\[
t^! \xHom{X}{Y} \to \xHom{t^! X}{t^! Y}
\]
is a weak CSS equivalence.

\item Let $\ul{\mathcal{K}}$ be the $\cat{\CSS}$-enriched category $t^! \argb{\ul{\cat{\Qcat}}}$ and let $\ul{t^!} : \ul{\mathcal{K}} \to \ul{\cat{\CSS}}$ be the canonical $\cat{\CSS}$-enrichment of $t^!$. The induced 2-functor
\[
\tau_1 \argb{\ul{t^!}} : \tau_1 \argb{\ul{\mathcal{K}}} \to \tau_1 \argb{\ul{\cat{\CSS}}}
\]
is a bicategorical equivalence.

\item The 2-functor $\tau_1 \argb{\ul{\mathcal{K}}} \to \bicat{\Qcat}$ induced by the adjunction counit $t_! t^! \hoto \id_{\cat{\SSet}}$ is an isomorphism of 2-categories.

\item The functor $t^! : \cat{\Qcat} \to \cat{\CSS}$ underlies a bicategorical equivalence $\bicat{\Qcat} \to \bicat{\CSS}$, and moreover $t^! \nv \cong \Nv$.
\end{enumerate}
\end{thm}
\begin{proof}
(i). Since $t^! : \cat{\SSet} \to \cat{\SSSet}$ is a right Quillen functor with respect to the Joyal model structure for quasicategories and the Rezk model structure for complete Segal spaces, $t^!$ must map quasicategories to complete Segal spaces. The preservation of small products is automatic because $t^!$ is a right adjoint and $\cat{\Qcat}$ is closed in $\cat{\SSet}$ under small products, and Ken Brown's lemma implies $t^!$ must preserve weak equivalences.

\bigskip\noindent
(ii). Since $\Ho t^! : \Ho \cat{\Qcat} \to \Ho \cat{\CSS}$ is a categorical equivalence, the canonical morphism $t^! \xHom{X}{Y} \to \xHom{t^! X}{t^! Y}$ must be an isomorphism in $\Ho \cat{\CSS}$, and therefore the corresponding morphism in $\cat{\CSS}$ must be a weak CSS equivalence.

\bigskip\noindent
(iii). We already know that $\tau_1 \argb{\ul{t^!}}$ induces equivalences of hom-categories, so it suffices to establish that the underlying $\parens{\Ho \cat{\Cat}}$-enriched functor is essentially surjective on objects. This in turn is a straightforward consequence of \autoref{cor:equivalences.in.CSS} and the fact that $t^! : \cat{\SSet} \to \cat{\SSSet}$ is half of a Quillen equivalence. 

\bigskip\noindent
(iv). By \autoref{cor:truncation.and.totalisation}, the functors $\tau_1 t_! t^! \xHom{X}{Y} \to \tau_1 \xHom{X}{Y}$ are isomorphisms for all quasicategories $X$ and $Y$, so the induced 2-functor $\tau_1 \argb{\ul{\mathcal{K}}} \to \bicat{\Qcat}$ is an isomorphism of 2-categories.

\bigskip\noindent
(v). This follows from claims (iii) and (iv).
\end{proof}

To conclude, we make note of a small generalisation of the above argument. Recall the following result of Toën:

\begin{thm}
\label{thm:Toen}
Let $\mathcal{M}$ be a combinatorial simplicial cartesian model category in which all objects are cofibrant and let $C^{\bullet}$ be a cosimplicial object in $\mathcal{M}$. If $\tuple{\mathcal{M}, C^{\bullet}}$ is a theory of $\tuple{\infty, 1}$-categories in the sense of Toën, then there exists a Quillen equivalence of the form below,
\[
R \dashv S : \mathcal{M} \to \cat{\SSSet}
\]
where $\cat{\SSSet}$ is equipped with the Rezk model structure for complete Segal spaces and $S$ is defined by the formula below:
\[
\parens{S M}_{n, \bullet} = \ulHom[\mathcal{M}]{C^n}{M}
\]
\end{thm}
\begin{proof} \openproof
See Théorème 5.1 in \citet{Toen:2005}.
\end{proof}

\begin{prop}
Let $\mathcal{M}$ be a cartesian model category in which all fibrant objects are cofibrant, let $\mathcal{M}_\mathrm{f}$ be the full subcategory of fibrant objects, and let
\[
R \dashv S : \mathcal{M} \to \cat{\SSSet}
\]
be a Quillen equivalence between $\mathcal{M}$ and the Rezk model structure for complete Segal spaces.
\begin{enumerate}[(i)]
\item The functor $S : \mathcal{M} \to \cat{\SSSet}$ restricts to a functor $S : \mathcal{M}_\mathrm{f} \to \cat{\CSS}$ that preserves small products and weak equivalences.

\item For all fibrant objects $X$ and $Y$ in $\mathcal{M}$, the canonical morphism
\[
S \xHom{X}{Y} \to \xHom{S X}{S Y}
\]
is a weak CSS equivalence.

\item Let $\ul{\mathcal{K}}$ be the $\cat{\CSS}$-enriched category $S \argb{\ul{\mathcal{M}_\mathrm{f}}}$ and let $\ul{S} : \ul{\mathcal{K}} \to \ul{\cat{\CSS}}$ be the canonical $\cat{\CSS}$-enrichment of $S$. The induced 2-functor
\[
\tau_1 \argb{\ul{S}} : \tau_1 \argb{\ul{\mathcal{K}}} \to \tau_1 \argb{\ul{\cat{\CSS}}}
\]
is a bicategorical equivalence.
\end{enumerate}
\end{prop}
\begin{proof}
The arguments in the proof of \autoref{thm:bicategorical.comparison.for.CSS} also work here.
\end{proof}

\begin{remark}
The underlying ordinary category of $\tau_1 \argb{\ul{\mathcal{K}}}$ need not be isomorphic to $\mathcal{M}_\mathrm{f}$; for this, it is necessary and sufficient that $S$ satisfy the hypotheses of \autoref{lem:enrichment.via.transport}. If we think of a morphism $1 \to X$ in $\mathcal{M}$ as being an object in the $\tuple{\infty, 1}$-category presented by $X$, this essentially amounts to demanding that $S$ preserve the set of objects.
\end{remark}

\begin{remark}
We can replace $\cat{\SSSet}$ (\resp $\cat{\CSS}$, \etc) in the proposition with $\cat{\SSet}$ (\resp $\cat{\Qcat}$, \etc). However, \citet{Joyal-Tierney:2007} have shown that there are Quillen equivalences between the two model categories in both directions, so it is also possible to formally reduce one version to the other by composing $R \dashv S$ with a suitable Quillen equivalence.
\end{remark}

\ifdraftdoc

\else
  \printbibliography
\fi

\end{document}